\documentclass[11pt]{article}
\usepackage[a4paper,margin=3cm]{geometry}
\usepackage[T1]{fontenc}      
\usepackage{lmodern}          
\usepackage{amsmath}          
\usepackage{amssymb}          
\usepackage{amsthm}           
\usepackage{mathtools}        
\usepackage{tikz}             
\usetikzlibrary{decorations.markings}
\usepackage{placeins}         
\usepackage{needspace}        
\usepackage[hidelinks]{hyperref} 
\hypersetup{
  pdftitle={Mixed partition functions are exactly the graph
    parameters of exponentially bounded edge-connection rank},
  pdfauthor={William Whistler},
  pdfsubject={Graph parameters, connection matrices, and mixed
    partition functions},
  pdfkeywords={graph parameter, connection matrix, partition
    function, mixed partition function, edge-colouring model,
    symmetric tensor category, super vector space, fibre functor,
    Schur functor}}

\newtheorem{theo}{Theorem}[section]
\newtheorem{lemm}[theo]{Lemma}
\newtheorem{coro}[theo]{Corollary}

\theoremstyle{definition}
\newtheorem{defi}[theo]{Definition}
\theoremstyle{plain}
\let\oldthebibliography\thebibliography
\renewcommand{\thebibliography}[1]{\oldthebibliography{#1}%
  \setlength{\itemsep}{0pt}}

\newcommand{\emptygraph}{\varnothing}
\newcommand{\freecircle}{\bigcirc}
\newcommand{\Frag}[1]{\mathfrak{F}_{#1}}          
\newcommand{\connmat}[2]{M_{#1,#2}}               
\newcommand{\glue}{\ast}                          
\newcommand{\strand}{I}                           
\newcommand{\halfedges}[2]{\delta_{#1}(#2)}       
\newcommand{\conncat}{\mathcal{C}_{f}}            
\newcommand{\cauchy}{\mathcal{D}_{f}}             
\newcommand{\unitobj}{\mathbf{1}}                 
\newcommand{\genX}{X}                             
\newcommand{\Kar}{\operatorname{Kar}}
\newcommand{\Add}{\operatorname{Add}}
\newcommand{\End}{\operatorname{End}}
\newcommand{\Hom}{\operatorname{Hom}}
\newcommand{\tr}{\operatorname{tr}}
\newcommand{\rk}{\operatorname{rk}}
\newcommand{\fib}{\omega}                         
\newcommand{\svect}{\mathrm{sVect}_{\mathbb{C}}}
\newcommand{\str}{\operatorname{str}}
\newcommand{\parityop}{\mathfrak{p}}              
\newcommand{\copair}{C}                           
\newcommand{\twistedC}{\widehat{C}}               
\newcommand{\mate}[1]{#1^{\flat}}                 
\newcommand{\symq}[1]{\varpi_{#1}}                
\newcommand{\SymV}{\operatorname{Sym}}
\newcommand{\parinv}{J}                           
\newcommand{\hook}[2]{H(#1,#2)}                   
\newcommand{\tzeta}[1]{\zeta_{#1}}                

\newcommand{\growthbase}{R_{f}}                   
\newcommand{\schur}[2]{S_{#1}(#2)}                
\newcommand{\onerow}[1]{s_{(#1)}}                 

\title{Mixed partition functions are exactly the graph parameters of
exponentially bounded edge-connection rank}
\author{William Whistler\\
\texttt{will@whistler.uk}}
\date{}

\begin{document}
\maketitle

\begin{abstract}
We prove a conjecture of Regts and Sevenster: a complex-valued graph
parameter $f$ with $f(\emptygraph)=1$ has exponentially bounded
edge-connection rank if and only if it is a mixed partition function.
The bound is exact: for a real number $R\ge1$, the connection ranks
satisfy $\rk\connmat{f}{t}\le R^{t}$ for all $t\ge0$ if and only if $f$
has a model on a super vector space $\mathbb{C}^{k|2\ell}$ with
$k+2\ell\le R$. Consequently the base of exponential growth of the
connection ranks is the least number of colours of a model, the two
dimensions of a minimal model are determined by $f$, and the
parameters with a model on a prescribed $\mathbb{C}^{k|2\ell}$ are
characterised. The proof organises fragments modulo the connection
kernel into a rigid symmetric tensor category whose morphism spaces
have the connection ranks as dimensions; the rank hypothesis and an
argument of Schrijver make its additive idempotent completion
semisimple, Deligne's theorem provides a fibre functor to super
vector spaces, and the resulting super tensor network is identified
with the Regts--Sevenster model exactly, circuit signs included. An
appendix shows that in a rigid symmetric $\mathbb{C}$-linear
category with $\End(\unitobj)=\mathbb{C}$, exponentially bounded
endomorphism growth makes the trace zeta function of every
endomorphism rational, with explicit degree bounds.

\medskip
\noindent\textit{MSC 2020:} 05C31, 18M05 (primary); 15A72, 05E05,
20C30 (secondary).

\noindent\textit{Keywords:} graph parameter; connection matrix;
partition function; mixed partition function; edge-colouring
model; symmetric tensor category; super vector space; fibre
functor; Schur functor.
\end{abstract}

\tableofcontents

\section{Introduction}\label{sec:intro}

Which graph parameters are partition functions? The question has
organised a substantial body of work since de la Harpe and Jones
\cite{dHJ93} introduced spin and vertex models on graphs as
generalisations of the models of statistical mechanics --- in
later terminology, vertex-colouring and edge-colouring models
respectively. Freedman,
Lov\'asz and Schrijver \cite{FLS07} characterised the partition
functions of vertex-colouring models by reflection positivity and
a rank condition on connection matrices; Szegedy \cite{Sze07} did
the same for edge-colouring models, with the complex case
characterised in \cite{DGLRS12}; and Schrijver characterised
the partition functions of spin models \cite{Sch13} and of
complex edge-colouring models \cite{Sch15} by rank growth.
The connection-matrix method is treated at length in Lov\'asz's
book \cite{Lov12}, and its categorical undercurrent --- that
connection matrices are Gram matrices of a category of fragments
--- goes back at least to Lov\'asz and Schrijver \cite{LS09}.

Regts and Sevenster \cite{RS21}, building on their symplectic
models \cite{RS17}, enlarged the edge-colouring world by a
fermionic sector: their \emph{mixed partition functions},
recalled in Section~\ref{subsec:mpf}, are partition functions of
edge-colouring models on a super vector space
$\mathbb{C}^{k|2\ell}$, and they include, for instance, every
evaluation of the characteristic polynomial
\cite[Prop.~10]{RS21}; the evaluation at zero, in particular, is
not the partition function of any ordinary edge-colouring model
\cite[Prop.~9]{RS21}. They proved that every mixed
partition function has exponentially bounded edge-connection rank
\cite[Thm~6]{RS21} and conjectured the converse. A 2017 extended
abstract had announced the characterisation as a theorem
\cite[Thm~3.1]{RS17ea}; the full paper withdraws the announcement
and states the converse as a conjecture
\cite[abstract and footnotes~1--2]{RS21}, the intended route
through the invariant theory of the orthosymplectic supergroup
having met an obstruction, discussed below.

This paper proves the conjecture. The slogan is that a purely
combinatorial exponential-rank condition forces a hidden
finite-dimensional super tensor model; the precise statement is
as follows.

\begin{theo}[Main Theorem]\label{theo:main}
A complex-valued graph parameter $f$, defined on finite multigraphs
with free circles permitted and normalised by $f(\emptygraph)=1$, has
exponentially bounded edge-connection rank if and only if it is a
mixed partition function. More precisely, for a real number
$R\ge 1$, we have $\rk \connmat{f}{t}\le R^{t}$ for all $t\ge 0$ if
and only if $f=p_{h}$ for a functional
$h\in(\SymV\mathbb{C}^{k}\otimes\Lambda\mathbb{C}^{2\ell})^{*}$
with $k+2\ell\le R$.
\end{theo}

The implication from mixed partition functions to exponentially
bounded rank is Regts and Sevenster's theorem, with
$R=\max(1,k+2\ell)$; our contribution is the converse, together
with the dimension bound, which matches their bound exactly. The
case $t=0$ of the hypothesis, together with $f(\emptygraph)=1$,
supplies the multiplicativity of $f$ over disjoint unions
(Lemma~\ref{lemm:mult}).

Equality of the two bounds has consequences, developed in
Section~\ref{sec:consequences}. Write
$\growthbase\coloneqq\sup_{t\ge1}(\rk\connmat{f}{t})^{1/t}$ for
the base of exponential growth of the connection ranks, the
\emph{growth base}; this is the edge-rank connectivity of
\cite{RS21}, with the value zero allowed.

\begin{theo}[Model dimensions]\label{theo:introdims}
Let $f(\emptygraph)=1$, suppose that $\rk\connmat{f}{t}\le R^{t}$
for all $t\ge0$ and some real $R\ge1$, and let $k$ and $2\ell$ be
the even and odd dimensions of the model that the proof produces.
Then $\growthbase=k+2\ell=\lim_{n\to\infty}(\rk\connmat{f}{2n})^{1/2n}$,
$f(\freecircle)=k-2\ell$, and the dimension pairs $(k',2\ell')$ of
the models of $f$ are exactly the pairs $(k+2r,\,2\ell+2r)$ with
$r\in\mathbb{N}$.
\end{theo}

Thus the least total number of colours of a model of $f$ is the
growth base, the two dimensions of a minimal model are read off
from $\growthbase$ and the free-circle value, and the attainable
totals are the numbers $\growthbase+4r$. The minimal dimension
pair is intrinsic to $f$; a minimal functional is not, even up to
changes of coordinates preserving the form
(Section~\ref{subsec:growthbase}). The same two invariants decide
which spaces carry a model at all: for a graph parameter with
$f(\emptygraph)=1$ and for $k',\ell'\in\mathbb{N}$,
Theorem~\ref{theo:prescribed} shows that $f$ has a model on
$\mathbb{C}^{k'|2\ell'}$ if and only if
\[
f(\freecircle)=k'-2\ell'
\qquad\text{and}\qquad
\rk\connmat{f}{t}\le(k'+2\ell')^{t}\quad(t\ge0),
\]
with $0^{0}=1$. Every model has superdimension $f(\freecircle)$,
while the least total dimension of a model is $\growthbase$; thus
the free-circle value and the growth of the connection ranks
recover the odd and even dimensions of a minimal model. For the
characteristic polynomial $G\mapsto\det(\theta I-A_{G})$, taken
on the multigraphs of this paper with the value $0$ on free
circles, these two invariants are $0$ and $4$, and the
Regts--Sevenster model on $\mathbb{C}^{2|2}$ is minimal for every
$\theta$ (Corollary~\ref{coro:charpoly}).

The proof runs as follows, and the sections follow it literally,
after Section~\ref{sec:prelim} has fixed conventions and recalled
the Regts--Sevenster definition of a mixed partition function
(Definition~\ref{defi:mpf}). From $f$ we construct a
\emph{connection category} $\conncat$: fragments modulo the
kernel of the connection pairing, organised into a rigid
symmetric monoidal category whose morphism spaces have the
connection ranks as their dimensions and whose trace pairings
are nondegenerate (Section~\ref{sec:conncat}). In its Cauchy
completion $\cauchy$, an argument of Schrijver
\cite[Prop.~4]{Sch15} shows that every nilpotent endomorphism has
trace zero: a nilpotent of nonzero trace would make the
endomorphism spaces of the tensor powers grow factorially, which
the rank bound forbids. Combined with the nondegeneracy of the
trace pairing this makes $\cauchy$ semisimple abelian, Deligne's
theorem \cite{Del02} provides an exact faithful symmetric tensor
functor to finite-dimensional super vector spaces, and
transporting the symmetric-group actions along that functor
yields the dimension bound (Section~\ref{sec:semisimple}).

The remainder of the proof is the concrete half. The fibre
functor yields a form, a copairing and one vertex tensor per
degree, and $f$ becomes the value of a super tensor network
(Section~\ref{sec:extraction}). A separate reconstruction theorem
then identifies that network, after a parity normalisation, with
the Regts--Sevenster sum exactly --- Eulerian subgraphs,
orientations, pairings, and the sign of $-1$ for every fermionic
circuit (Section~\ref{sec:circuit}). Section~\ref{sec:consequences}
draws the consequences: the dimensions of all models of $f$ are
described, the model the proof produces being the smallest, the
ordinary and skew models are located inside the mixed ones, with
Schrijver's characterisation of the former recovered exactly, and
Corollary~\ref{coro:rectangle} is compared with a conjecture from
Sevenster's thesis on the models of a fixed $\mathbb{C}^{k|2\ell}$
(Section~\ref{subsec:sevenster}); Section~\ref{subsec:formal}
records what has been formalised.
Appendix~\ref{app:zeta} strengthens the nilpotent-trace step to
the rationality of trace zeta functions. The connection-category and reconstruction
arguments are developed self-containedly; beyond them the proof
uses Deligne's theorem and the standard representation theory of
symmetric groups, and Section~\ref{sec:consequences} uses the hook
criterion of Berele and Regev.

The obstruction met by Regts and Sevenster is the failure of
complete reducibility, in general, for the finite-dimensional
representations of $\mathfrak{osp}$ (``not reductive'' in their
words), which blocks the expected argument through orthosymplectic
invariant theory; the companion manuscript in which that route was
to be completed has, to our knowledge, not appeared. The
categorical route never touches that invariant theory:
semisimplicity is proved for the category obtained by quotienting
the fragments by every relation invisible to graph closures,
rather than assumed for an ambient category of orthosymplectic
representations, where it fails in general. The quotient makes
the trace pairings nondegenerate, and the growth hypothesis turns
that nondegeneracy into semisimplicity.

The contribution is thus an integrated argument in three parts: a
categorical route around the obstruction, through the connection
quotient, Schrijver's argument and Deligne's theorem; an exact
comparison of rank growth with dimension, in which Deligne's
functor places the symmetric-group blocks acting on the model
inside the connection category, where the rank hypothesis bounds
them (Lemma~\ref{lemm:diminequality}); and the sign-sensitive
identification of the abstract fibre functor's output with the
exact normalisations of \cite[Definition~5]{RS21}, the
reconstruction theorem of Section~\ref{sec:circuit}. The
semisimplicity step is Schrijver's argument
\cite[Propositions~4 and~5]{Sch15}; Etingof and Penneys have since
proved the vanishing of nilpotent traces in much greater
generality \cite{EP26}. Appendix~\ref{app:zeta} gives a direct
objectwise proof of the rationality of trace zeta functions, with
rectangular degree bounds, together with an explicit sufficient
criterion from endomorphism growth;
Section~\ref{subsec:appliterature} explains the antecedents. The
converse Holant
theorems of Young \cite{You25} and of Cai and Young \cite{CY26}
(see also \cite{CG22}) concern the equivalence of signature sets
under a transformation group, a different question.

The proof
is existential: it provides no procedure that reconstructs the
functional $h$ from the values of $f$, Deligne's functor being
non-constructive; the two dimensions of the model it produces are,
by contrast, determined by $f$ alone
(Theorem~\ref{theo:modeldims}).

Both implications of the Main Theorem have been machine-checked in
Lean~4 over mathlib \cite{Whi26}; Section~\ref{subsec:formal}
records what the development covers and how its conventions relate
to those of the paper.

\section[Mixed partition functions, fragments, and half-edge
conventions]{Mixed partition functions, fragments, and half-edge
\mbox{conventions}}\label{sec:prelim}

The purpose of this section is twofold: to fix our conventions for
multigraphs, fragments and gluing, and to state the Regts--Sevenster
definition of a mixed partition function together with the hypothesis
class of the Main Theorem.

\subsection{Graphs, half-edges, and parameters}\label{subsec:graphs}

A \emph{graph} is a finite multigraph: loops, parallel edges, and
\emph{free circles} --- closed edges incident with no vertex,
isomorphic to the circle $\freecircle$ --- are all permitted. We write
$\emptygraph$ for the empty graph and $G\sqcup H$ for disjoint union.
A \emph{graph parameter} is a function from isomorphism classes of
graphs to $\mathbb{C}$. We write $\mathbb{N}=\mathbb{Z}_{\ge0}$.

Every edge other than a free circle consists of two \emph{half-edges},
one at each endpoint; the two half-edges of a loop are distinct and
lie at the same vertex. The degree $\deg(v)$ counts half-edges, so
that a loop contributes two, and for $F\subseteq E$ we write
$\halfedges{F}{v}$ for the set of half-edges at $v$ belonging to edges
of $F$. Throughout this paper, orders, orientations and pairings of edges
are carried by the half-edges, while colours are assigned to edges
and inherited by both half-edges; this convention matters at
loops and parallel edges, and we shall rely on it without further
comment. Formally, a graph is a finite vertex set, a finite
half-edge set, an incidence map, a perfect matching of half-edges
into edges, and a nonnegative integer recording the free-circle
components, which carry no half-edges.

\subsection{Fragments and gluing}\label{subsec:fragments}

Intuitively, a fragment is a graph with some dangling edge-ends, which
may later be glued to the dangling ends of another fragment.
Formally, a \emph{$t$-fragment} is a graph together with exactly $t$
distinguished vertices of degree one, labelled $1,\dotsc,t$; the edge
at the $i$-th labelled vertex is the \emph{open end} $i$. We write
$\Frag{t}$ for the set of isomorphism classes of $t$-fragments, so
that $\Frag{0}$ is the set of graphs. To \emph{glue} two open ends,
delete the two labelled vertices and unify their edges into a single
edge; for $F,G\in\Frag{t}$, the \emph{full gluing}
$F\glue G\in\Frag{0}$ glues the $i$-th open end of $F$ to the $i$-th
open end of $G$ for every $i$. Note that gluing to each other the
two open ends of the $2$-fragment consisting of a single edge
joining its two labelled vertices produces a free circle, as does
the full gluing of two copies of that fragment; this is
the reason free circles are admitted from the outset. We show two
examples, including this degenerate case, in
Figure~\ref{fig:gluing}.

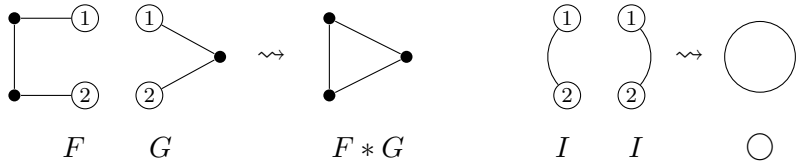
\begin{figure}[htbp]
\centering
\begin{tikzpicture}[scale=0.85,
  vert/.style={circle,fill,inner sep=1.6pt},
  lab/.style={circle,draw,inner sep=1.2pt,font=\scriptsize}]
  \begin{scope}
    \node[vert] (a) at (0,0.6) {};
    \node[vert] (b) at (0,-0.6) {};
    \draw (a) -- (b);
    \node[lab] (l1) at (1.1,0.6) {1};
    \node[lab] (l2) at (1.1,-0.6) {2};
    \draw (a) -- (l1); \draw (b) -- (l2);
    \node[lab] (m1) at (2.1,0.6) {1};
    \node[lab] (m2) at (2.1,-0.6) {2};
    \node[vert] (c) at (3.2,0) {};
    \draw (m1) -- (c); \draw (m2) -- (c);
    \node at (1.6,-1.4) {$F$\hspace{2.2em}$G$};
    \node at (4.0,0) {$\leadsto$};
    \node[vert] (a2) at (4.9,0.6) {};
    \node[vert] (b2) at (4.9,-0.6) {};
    \node[vert] (c2) at (6.1,0) {};
    \draw (a2) -- (b2); \draw (a2) -- (c2); \draw (b2) -- (c2);
    \node at (5.5,-1.4) {$F\glue G$};
  \end{scope}
  \begin{scope}[xshift=8.6cm]
    \node[lab] (p1) at (0,0.6) {1};
    \node[lab] (p2) at (0,-0.6) {2};
    \draw (p1) to[bend right=40] (p2);
    \node[lab] (q1) at (1.0,0.6) {1};
    \node[lab] (q2) at (1.0,-0.6) {2};
    \draw (q1) to[bend left=40] (q2);
    \node at (0.5,-1.4) {$\strand\qquad\strand$};
    \node at (1.9,0) {$\leadsto$};
    \draw (3.0,0) circle (0.55);
    \node at (3.0,-1.4) {$\freecircle$};
  \end{scope}
\end{tikzpicture}
\caption{Gluing two fragments; the free-circle-creating case.}
\label{fig:gluing}
\end{figure}

Two further operations on fragments, composition and the identity
strand, are introduced at the start of Section~\ref{sec:conncat},
where their elementary properties are established.

\subsection{Connection matrices and the hypothesis
class}\label{subsec:connmat}

The \emph{edge-connection matrix} $\connmat{f}{t}$ of a graph
parameter $f$ is the $\Frag{t}\times\Frag{t}$ matrix with entries
\[
\connmat{f}{t}(F,G)\coloneqq f(F\glue G),
\]
and its rank $\rk\connmat{f}{t}$ is the supremum of the ranks of its
finite submatrices. The hypotheses of the Main Theorem are:
\begin{itemize}
\item[(H1)] $f(\emptygraph)=1$;
\item[(H2)] there is a real number $R\ge 1$ with
  $\rk\connmat{f}{t}\le R^{t}$ for all $t\ge 0$.
\end{itemize}

\subsection{Mixed partition functions}\label{subsec:mpf}

The definition of a mixed partition function recalled in this
subsection is Definition~5 of \cite{RS21}, restated in the
half-edge conventions of Section~\ref{subsec:graphs} as
Definition~\ref{defi:mpf} below.

Fix $k,\ell\in\mathbb{N}$ and let $V=V_{0}\oplus V_{1}$ be a super
vector space with $\dim V_{0}=k$ and $\dim V_{1}=2\ell$, with
distinguished bases $e_{1},\dotsc,e_{k}$ of $V_{0}$ and
$\xi_{1},\dotsc,\xi_{2\ell}$ of $V_{1}$. To avoid confusion we write
$\xi_{i},\eta_{i}$ where \cite{RS21} write $f_{i},g_{i}$, reserving
the letter $f$ for graph parameters throughout this paper; here
\[
\eta_{i}\coloneqq
\begin{cases}
-\xi_{i+\ell}& i\le \ell,\\
\hphantom{-}\xi_{i-\ell}& i>\ell,
\end{cases}
\]
so that each $\eta_{i}$ is a signed basis vector. Let
$h\in(\SymV V_{0}\otimes\Lambda V_{1})^{*}$. Regts and Sevenster
permit an arbitrary such functional; since the exterior degree
appearing at a vertex in the sum below is $\deg_{F}(v)$, which is
even whenever $F$ is Eulerian, the odd exterior part of $h$ never
contributes, and we may therefore impose, without loss of
generality, that $h$ vanishes on every component of odd exterior
degree.

Call $F\subseteq E(G)$ \emph{Eulerian} if $\deg_{F}(v)$ is even for
every vertex $v$. For Eulerian $F$, choose an Eulerian orientation
of $F$ together with a compatible \emph{local pairing}
$\kappa=(\kappa_{v})_{v}$, which pairs at each vertex $v$ each
incoming half-edge of $\halfedges{F}{v}$ with an outgoing one; we
list each pair as $(a_{1},a_{2})$ with $a_{1}$ incoming. Following
the pairs from edge to edge decomposes $F$ into directed
\emph{$\kappa$-circuits}, and we write $c(\kappa)$ for their
number.

\begin{defi}[{Mixed partition function; Regts--Sevenster
\cite[Definition~5]{RS21}}]\label{defi:mpf}
With $k$, $\ell$ and $h$ as above, the \emph{mixed partition
function} of $h$ is the graph parameter
\begin{multline}\label{eq:mpf}
p_{h}(G)\coloneqq
\sum_{\substack{F\subseteq E(G)\\ F\text{ Eulerian}}}
(-1)^{c(\kappa)}
\sum_{\psi\colon E\setminus F\to[k]}\;
\sum_{\varphi\colon F\to[2\ell]}\\
\prod_{v\in V(G)}
h\Bigl(
\bigodot_{a\in\halfedges{E\setminus F}{v}}\hspace{-1ex}e_{\psi(a)}
\;\otimes
\bigwedge_{(a_{1},a_{2})\in\kappa_{v}}\hspace{-1ex}
\xi_{\varphi(a_{1})}\wedge\eta_{\varphi(a_{2})}
\Bigr),
\end{multline}
for a graph $G$ without free circles, extended to all graphs by
$p_{h}(\freecircle)\coloneqq k-2\ell$ and multiplicativity over
disjoint union. A graph parameter is a \emph{mixed partition
function} if it equals $p_{h}$ for some $k$, $\ell$ and $h$.
\end{defi}

Here $\psi$ and $\varphi$ colour edges, and each
half-edge inherits the colour of its edge. The wedge over the
pairs of $\kappa_{v}$ is unambiguous: each factor
$\xi_{\varphi(a_{1})}\wedge\eta_{\varphi(a_{2})}$ has even
exterior degree, so the order of the pairs is immaterial and the
displayed argument of $h$ is a well-defined element of
$\SymV V_{0}\otimes\Lambda V_{1}$. The value is independent
of the choice of orientation and compatible local pairing:
\cite[Proposition~3]{RS21} proves this for the skew sector, and
the mixed sum reduces to it, one colouring $\psi$ at a time, in
the discussion preceding their Definition~5.

The value on the free circle is Regts and Sevenster's: they admit
the circle, a single edge with an empty vertex set, as a graph,
and their formula evaluates it to $k-2\ell$
\cite[remark following Definition~5]{RS21}; we separate the
circle-free case from the general one because our half-edge
formalism counts free circles separately. Every model of $f$
therefore has superdimension $f(\freecircle)$, which is one of the
two quantities through which Theorem~\ref{theo:modeldims} reads
the dimensions of a minimal model off $f$.

Regts and Sevenster proved that
$\rk\connmat{p_{h}}{t}\le(k+2\ell)^{t}$ for every $t\ge 0$
\cite[Theorem~6]{RS21}, with $0^{0}=1$ in the degenerate case
$k=2\ell=0$; thus every mixed partition function satisfies (H2)
with $R=\max(1,k+2\ell)$ --- the maximum matters only for the
degenerate model --- and (H1) is immediate, the empty graph
carrying the empty product. The Main Theorem is the converse, with
the same constant.

The definition is somewhat intricate, and so we now give an
explicit example, to which we shall return. Take $k=2$ and
$\ell=1$, fix $\theta\in\mathbb{C}$, and let
$h=h(\theta)$ be determined by
\[
h(e_{1}^{\odot i})=\theta,\qquad
h(e_{1}^{\odot i}\odot e_{2})=\sqrt{-1},\qquad
h(e_{1}^{\odot i}\otimes \xi_{1}\wedge\eta_{1})=1
\qquad(i\ge 0),
\]
and by $h=0$ on all basis elements outside the span of these. Regts
and Sevenster showed \cite[Proposition~10]{RS21} that, for graphs
without free-circle components, $p_{h(\theta)}$ is the
characteristic polynomial $G\mapsto\det(\theta I-A_{G})$, where
the adjacency matrix $A_{G}$ counts a loop twice on the diagonal;
since $k-2\ell=0$ here, $p_{h(\theta)}$ vanishes on every graph
that does contain a free circle.

We evaluate $p_{h(\theta)}$ on the graph $L$ with one vertex and one
loop; see Figure~\ref{fig:definitionfive}. The loop has two
half-edges at the single vertex $v$, so $\deg(v)=2$, and both
$F=\emptyset$ and $F=\{\text{loop}\}$ are Eulerian. For
$F=\emptyset$, the vertex factor is
$h(e_{\psi}\odot e_{\psi})$: the colouring $\psi=1$ contributes
$h(e_{1}^{\odot 2})=\theta$, while $\psi=2$ contributes
$h(e_{2}^{\odot 2})=0$. For $F=\{\text{loop}\}$, the pairing
$\kappa_{v}$ joins the loop's two half-edges into a single
$\kappa$-circuit, so $c(\kappa)=1$; the
colouring $\varphi=j$ contributes $h(\xi_{j}\wedge\eta_{j})$, and
since
\[
\xi_{1}\wedge\eta_{1}
=-\xi_{1}\wedge\xi_{2}
=\xi_{2}\wedge\eta_{2},
\]
both odd colours contribute through the same basis vector of
$\Lambda^{2}V_{1}$, each with value $1$. In total
\[
p_{h(\theta)}(L)
=\theta+(-1)^{1}\cdot(1+1)
=\theta-2
=\det\bigl(\theta I-A_{L}\bigr),
\]
as required, since $A_{L}=(2)$. This small computation has already
exercised the two incidences of a loop, the Eulerian condition,
the circuit sign, the distinction between a loop and a free
circle, and the fact that the $\eta$-convention makes distinct odd
colourings contribute through a common basis vector; all five
return in Sections~\ref{sec:extraction} and~\ref{sec:circuit}.

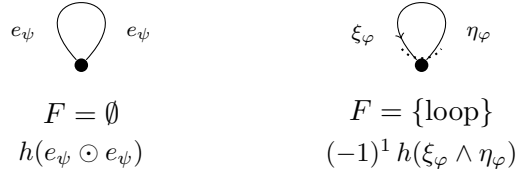
\begin{figure}
\centering
\begin{tikzpicture}[scale=0.9,
  vert/.style={circle,fill,inner sep=1.8pt}]
  \begin{scope}
    \node[vert] (v) at (0,0) {};
    \draw (v) .. controls (-1.1,1.2) and (1.1,1.2) .. (v);
    \node[font=\scriptsize] at (-0.85,0.45) {$e_{\psi}$};
    \node[font=\scriptsize] at (0.85,0.45) {$e_{\psi}$};
    \node at (0,-0.7) {$F=\emptyset$};
    \node[font=\small] at (0,-1.3) {$h(e_{\psi}\odot e_{\psi})$};
  \end{scope}
  \begin{scope}[xshift=5cm]
    \node[vert] (w) at (0,0) {};
    \draw[postaction={decorate},decoration={markings,
      mark=at position 0.2 with {\arrow{<}}}]
      (w) .. controls (-1.1,1.2) and (1.1,1.2) .. (w);
    \node[font=\scriptsize] at (-0.85,0.45) {$\xi_{\varphi}$};
    \node[font=\scriptsize] at (0.85,0.45) {$\eta_{\varphi}$};
    \draw[dotted,thick] (-0.29,0.24) to[bend right=55] (0.29,0.24);
    \node at (0,-0.7) {$F=\{\text{loop}\}$};
    \node[font=\small] at (0,-1.3)
      {$(-1)^{1}\,h(\xi_{\varphi}\wedge\eta_{\varphi})$};
  \end{scope}
\end{tikzpicture}
\caption{The loop graph $L$, its two Eulerian subsets, and the
pairing $\kappa$.}
\label{fig:definitionfive}
\end{figure}

\section{The connection category and its Cauchy
completion}\label{sec:conncat}

We now organise fragments and gluing into a symmetric monoidal
category $\conncat$, whose morphism spaces have the connection ranks
as their dimensions, and complete it to a category $\cauchy$ in which
the remainder of the proof takes place. The category theory involved is elementary and self-contained:
every trace identity below is proved by exhibiting an equality of
glued graphs, and $f$ is applied only afterwards. For readers of a
combinatorial bent, the following dictionary may serve as a guide.

\begin{center}
\begin{tabular}{ll}
number of open ends & object\\
fragment & morphism\\
gluing & composition\\
disjoint union & tensor product\\
full closure & trace\\
connection pairing & trace pairing
\end{tabular}
\end{center}

We call $\conncat$ the \emph{connection category} of $f$.

\subsection{Composition of fragments}\label{subsec:composition}

Fragments compose by partial gluing, and we fix the categorical
order once and for all: a fragment $F\in\Frag{s+t}$ is read as a
morphism from its first $s$ open ends (its \emph{inputs}) to its
last $t$ (its \emph{outputs}), and for $F\in\Frag{s+t}$ and
$G\in\Frag{t+u}$ the \emph{composite} $G\circ F\in\Frag{s+u}$ ---
first $F$, then $G$, written right to left --- glues the last $t$
open ends of $F$ to the first $t$ open ends of $G$, in order. The
inputs of $G\circ F$ are the inputs of $F$ with their labels, and
its outputs are the outputs of $G$, relabelled
$s+1,\dotsc,s+u$. The \emph{strand}
$\strand\in\Frag{2}$ is the fragment with a single edge joining its
two labelled vertices, and more generally, for a permutation $\pi$ of
$[t]$, the \emph{permutation fragment} $P_{\pi}\in\Frag{t+t}$
consists of $t$ disjoint strands joining open end $i$ to open end
$t+\pi(i)$. Whenever a disjoint union of fragments is read as a
morphism, it carries the input--input--output--output block
relabelling made explicit in Section~\ref{subsec:categorydef}.

\begin{lemm}\label{lemm:gluingfacts}
The following hold for all fragments of compatible arities.
\begin{enumerate}
\item[(i)] Composition is associative:
  $(H\circ G)\circ F=H\circ(G\circ F)$.
\item[(ii)] The strand is an identity: writing $\strand^{t}$ for
  $t$ disjoint strands, $F\circ\strand^{s}=F=\strand^{u}\circ F$ for
  $F\in\Frag{s+u}$.
\item[(iii)] Composition and disjoint union interchange: for
  $F\in\Frag{s+t}$, $G\in\Frag{t+u}$, $F'\in\Frag{s'+t'}$ and
  $G'\in\Frag{t'+u'}$,
  $(G\sqcup G')\circ(F\sqcup F')=(G\circ F)\sqcup(G'\circ F')$.
\item[(iv)] Gluing commutes with simultaneous relabelling of the
  interface, and $P_{\pi}\circ P_{\sigma}=P_{\pi\sigma}$, where
  $\pi\sigma$ denotes the composite function ($\sigma$ first).
\end{enumerate}
\end{lemm}

\begin{proof}
All four statements are instances of one description of iterated
gluing. A gluing is the quotient of the edge multiset of a disjoint
union along a pairing of open ends: each matched pair of pendant
edges is unified into one edge, chains of matched strands collapse
into single edges, and closed chains collapse into free circles. In
an iterated composite each open end participates in at most one
interface pair, so the union of the interface pairings is itself a
single pairing on ends, and unifying iteratively in any order equals
unifying simultaneously along that union; this proves (i). For (ii),
gluing a strand's end to an open end $o$ of $F$ deletes the two
labelled vertices and unifies the strand's edge with the edge of
$o$, yielding a single pendant edge ending at the strand's other
labelled vertex, which is $F$ with that open end relabelled. For
(iii), the interface pairing of the left-hand side splits as the
disjoint union of the blockwise pairings, and quotienting a disjoint
union along a split pairing is the disjoint union of the quotients.
Finally (iv): a relabelling is an isomorphism of the data the
pairing is defined on, and composing permutation fragments
concatenates strands, which by (ii) collapse to the strands of the
composite permutation.
\end{proof}

\subsection{The connection category}\label{subsec:categorydef}

For the remainder of the paper, $f$ denotes a graph parameter
satisfying (H1) and (H2) with rank bound $R$. We begin by recording
why (H2) is imposed at $t=0$ as well: that case is where the
multiplicativity of $f$ over disjoint unions, which the Main
Theorem does not assume, comes from. For a parameter normalised at
the empty graph, rank at most one of the connection matrix of arity
zero is equivalent to multiplicativity; this is standard
\cite[Exercise~4.4]{Lov12}, and the rank hypothesis at arity zero
plays the same role in the characterisation of Freedman, Lov\'asz
and Schrijver \cite{FLS07}, as Lov\'asz remarks
\cite[remark following Theorem~5.54]{Lov12}. Without the case
$t=0$ the Main Theorem would fail: for
$f(G)\coloneqq(1+2^{v(G)})/2$, with $v(G)$
the number of vertices of $G$, two $t$-fragments $F$ and $G$ with
$u(F)$ and $u(G)$ unlabelled vertices have
$f(F\glue G)=\tfrac12+\tfrac12\,2^{u(F)}2^{u(G)}$, so every
connection matrix has rank at most two and
$\rk\connmat{f}{t}\le2^{t}$ for all $t\ge1$, while
$f(\emptygraph)=1$ and
$f(K_{1}\sqcup K_{1})=\tfrac52\ne\tfrac94=f(K_{1})^{2}$.

\begin{lemm}\label{lemm:mult}
For all graphs $G$ and $H$ we have $f(G\sqcup H)=f(G)f(H)$.
\end{lemm}

\begin{proof}
At $t=0$, hypothesis (H2) reads $\rk\connmat{f}{0}\le R^{0}=1$, and
$\connmat{f}{0}(G,H)=f(G\sqcup H)$. The row of $\emptygraph$ is
$(f(H))_{H}$, which is nonzero since $f(\emptygraph)=1$ by (H1).
A matrix of rank at most one with a nonzero row has every row a
scalar multiple of it: $f(G\sqcup H)=c_{G}f(H)$ for some
$c_{G}\in\mathbb{C}$, and setting $H=\emptygraph$ gives
$c_{G}=f(G)$.
\end{proof}

Let $\mathbb{C}^{(\Frag{t})}$ denote the vector space of finite
formal $\mathbb{C}$-combinations of $t$-fragments, and let
\[
\Theta_{t}\colon\mathbb{C}^{(\Frag{t})}\to
\bigl(\mathbb{C}^{(\Frag{t})}\bigr)^{*},
\qquad
\Theta_{t}(F)\coloneqq\bigl(G\mapsto f(F\glue G)\bigr),
\]
extended linearly, with kernel
$N_{t}\coloneqq\ker\Theta_{t}$. The quotient
$\mathbb{C}^{(\Frag{t})}/N_{t}$ has dimension exactly
$\rk\connmat{f}{t}$: it is isomorphic to the image of $\Theta_{t}$,
i.e.\ to the row space of the connection matrix, and any finite
linearly independent family of rows remains independent after
restriction to a suitable finite set of columns, so the row-space
dimension equals the supremum of the finite-submatrix ranks. By
(H2) this dimension is at most $R^{t}$; this is how the rank
hypothesis will enter every dimension count below.

\begin{lemm}\label{lemm:ideal}
The kernels $N$ form an ideal for both operations, in the following
sense.
\begin{enumerate}
\item[(a)] If $x\in N_{s+t}$ and $w$, $y$ are fragments of
  compatible arities, then $w\circ x\circ y\in N$.
\item[(b)] If $x\in N_{s+t}$ and $z\in\Frag{s'+t'}$, then
  $x\sqcup z\in N_{(s+s')+(t+t')}$.
\end{enumerate}
\end{lemm}

\begin{proof}
For (a), let $x\in N_{s+t}$, $y\in\Frag{r+s}$ and
$w\in\Frag{t+u}$, so that
$w\circ x\circ y\in\mathbb{C}^{(\Frag{r+u})}$, and let
$G\in\Frag{r+u}$. Form the $(s+t)$-fragment $G'$ by gluing the $r$
inputs of $y$ to the first $r$ open ends of $G$ and the $u$ outputs
of $w$ to the last $u$ open ends of $G$, labelling the $s$ outputs
of $y$ as $1,\dotsc,s$ and the $t$ inputs of $w$ as
$s+1,\dotsc,s+t$. For each fragment $x'$ in the support of $x$,
the full closures $(w\circ x'\circ y)\glue G$ and $x'\glue G'$
perform the same simultaneous unification, by
Lemma~\ref{lemm:gluingfacts}(i); extending linearly,
$f\bigl((w\circ x\circ y)\glue G\bigr)=f(x\glue G')=0$. For (b),
given a closure $G$ of
$x\sqcup z$, first glue the open ends of $z$ into $G$, obtaining a
fragment $G_{z}$ of arity $s+t$, possibly with extra closed
components; then
$f\bigl((x\sqcup z)\glue G\bigr)=f(x\glue G_{z})=0$.
\end{proof}

We may now define the category. The \emph{connection category}
$\conncat$ has objects the natural numbers, the object $t$ being
written $\genX^{\otimes t}$; its morphism spaces are
\[
\Hom(\genX^{\otimes s},\genX^{\otimes t})
\coloneqq\mathbb{C}^{(\Frag{s+t})}/N_{s+t},
\]
a fragment being read as a morphism from its first $s$ open ends to
its last $t$; composition is induced by the composition of
fragments, which descends to the quotients by
Lemma~\ref{lemm:ideal}(a), and the identity of $\genX^{\otimes t}$
is the class of $\strand^{t}$, by Lemma~\ref{lemm:gluingfacts}(ii).
The tensor product is addition on objects and is induced by
disjoint union on morphisms, with the block relabelling that the
input--output reading demands: for $F\in\Frag{s+t}$ and
$F'\in\Frag{s'+t'}$, the disjoint union is read as a morphism with
inputs the inputs of $F$ followed by those of $F'$ and outputs the
outputs of $F$ followed by those of $F'$; concretely, $F$-input
$i\mapsto i$, $F'$-input $i\mapsto s+i$, $F$-output
$j\mapsto s+s'+j$, and $F'$-output $j\mapsto s+s'+t+j$. This
descends by Lemma~\ref{lemm:ideal}(b); it is
strictly associative and unital with unit
$\unitobj\coloneqq\genX^{\otimes 0}$, since disjoint union of
isomorphism classes is so. Thus $\conncat$ is a
$\mathbb{C}$-linear strict monoidal category.

\begin{lemm}\label{lemm:simpleunit}
$\End(\unitobj)=\mathbb{C}\cdot[\emptygraph]$.
\end{lemm}

\begin{proof}
By construction
$\dim\End(\unitobj)=\rk\connmat{f}{0}$, which is $1$ by (H1) and
(H2), and $[\emptygraph]\ne 0$ since
$f(\emptygraph\glue\emptygraph)=1$.
\end{proof}

We henceforth identify $\End(\unitobj)$ with $\mathbb{C}$ via
$[\emptygraph]\mapsto 1$; every closed graph $G$ thereby represents
the scalar $f(G)$. Indeed $[G]=f(G)[\emptygraph]$: for every
closed graph $H$,
$f\bigl((G-f(G)\,\emptygraph)\sqcup H\bigr)
=f(G\sqcup H)-f(G)f(H)=0$ by Lemma~\ref{lemm:mult}, so
$G-f(G)\,\emptygraph\in N_{0}$.

The category is symmetric. The braiding of $\genX$ with itself is
the class of $P_{(1\,2)}$, and more generally the classes of the
permutation fragments realise every permutation of tensor factors;
by Lemma~\ref{lemm:gluingfacts}(iv) they compose as the permutations
do. Since the monoidal structure is strict, the hexagon axioms
reduce to identities among permutation fragments of the form
$P_{\pi\sqcup\mathrm{id}}\circ P_{\mathrm{id}\sqcup\pi}$, which hold
by the same composition rule; naturality of the braiding is
Lemma~\ref{lemm:gluingfacts}(iv) applied to the relabelling the
braiding performs; and the braiding squares to the identity.

The object $\genX$ is self-dual. Both the coevaluation
$\mathrm{coev}\in\Hom(\unitobj,\genX^{\otimes2})$ and the evaluation
$\mathrm{ev}\in\Hom(\genX^{\otimes2},\unitobj)$ are the class of the
strand $\strand$, read with both open ends as outputs and both as
inputs respectively; the snake identities are literal strand
unifications, viz.\ instances of Lemma~\ref{lemm:gluingfacts}(ii).
Consequently every object of $\conncat$ is self-dual: the
evaluation and coevaluation of $\genX^{\otimes t}$ are the classes
of $t$ parallel strands, the $i$-th end of one block joined to the
$i$-th end of the other, and the snake identities again reduce to
strand unifications. The \emph{trace} of
$u\in\End(\genX^{\otimes t})$ is its full strand closure --- the
scalar obtained by gluing the $i$-th output back to the $i$-th
input for every $i$. On the class of a fragment the trace is the
$f$-value of its strand closure, and on an arbitrary morphism it
is obtained by linear extension. In particular
\[
\tr(\mathrm{id}_{\genX})=f(\freecircle),
\]
the categorical dimension of $\genX$, closing a single strand
producing the free circle.

Finally, closed graphs decompose into stars. Let
$\sigma_{d}\in\Hom(\unitobj,\genX^{\otimes d})$ denote the class of
the \emph{$d$-star}: a single vertex with $d$ pendant edges, its
open ends labelled $1,\dotsc,d$; in particular $\sigma_{0}$ is
the isolated vertex. For a closed graph $G$ without free
circles, fix orders on its vertices and edges, and for each edge
an order of its two half-edges; then in $\conncat$
\begin{equation}\label{eq:stars}
[G]=
\Bigl(\bigotimes_{e\in E}\mathrm{ev}\Bigr)
\circ P_{G}\circ
\Bigl(\bigotimes_{v\in V}\sigma_{\deg v}\Bigr),
\end{equation}
where $P_{G}$ is the permutation fragment matching each star leg to
its half-edge. Indeed, the right-hand side glues the stars' pendant
edges in pairs along $E$, by
Lemma~\ref{lemm:gluingfacts}(i) and (iii), and each matched pair
unifies into the single edge $e$, reconstructing $G$; a loop sends
two legs of the same star to one evaluation, which is what the
half-edge convention of Section~\ref{subsec:graphs} provides for.

\subsection{The trace calculus}\label{subsec:tracecalc}

The following identities are standard consequences of rigidity and
symmetry; in $\conncat$ each asserts that two closed graphs are
equal. We abbreviate $uv\coloneqq u\circ v$.

\begin{lemm}\label{lemm:tracecalc}
Let $u$ and $v$ be morphisms of $\conncat$ of appropriate types.
\begin{enumerate}
\item[(a)] $\tr(uv)=\tr(vu)$ for
  $u\colon Y\to Z$ and $v\colon Z\to Y$.
\item[(b)] $\tr_{Y\otimes Z}(u\otimes v)=\tr(u)\tr(v)$ for
  $u\in\End(Y)$, $v\in\End(Z)$.
\item[(c)] $\tr\bigl(c_{Z,Z}\circ(u\otimes v)\bigr)=\tr(uv)$ for
  $u,v\in\End(Z)$, where $c$ is the braiding.
\end{enumerate}
\end{lemm}

\begin{proof}
By bilinearity it suffices to prove each identity when the
morphisms are represented by individual fragments.
(a) Both traces are the $f$-value of the same closed graph: the
strand closure of $u\circ v$ glues every open end of $u$ to the
matching open end of $v$ --- the composition interface directly, and
the outer legs through the closing strands, which merge away by
Lemma~\ref{lemm:gluingfacts}(ii) --- and the closure of $v\circ u$
performs the identical set of end-pairings.
(b) The closure of a disjoint union is the disjoint union of the
closures, and $f$ is multiplicative by Lemma~\ref{lemm:mult}.
(c) Closing $c\circ(u\otimes v)$ connects the output of $u$ through
the swap to the input of $v$, and the output of $v$ back to the
input of $u$; the resulting closed graph is the cyclic gluing of $u$
and $v$, which is exactly the closure of $u\circ v$, the
interface strands again merging away.
\end{proof}

\begin{lemm}\label{lemm:negligible}
If $u\in\Hom(\genX^{\otimes s},\genX^{\otimes t})$ satisfies
$\tr(uv)=0$ for every
$v\in\Hom(\genX^{\otimes t},\genX^{\otimes s})$, then $u=0$.
\end{lemm}

\begin{proof}
The trace closure of $u$ against a fragment $v$ differs from the
defining pairing $u\glue G$ of Section~\ref{subsec:categorydef} only
by a fixed relabelling of the $s+t$ open ends. Relabelling is a
bijection on fragments, so the two families of test values coincide,
and the hypothesis places a lift of $u$ in $N_{s+t}$.
\end{proof}

In the language of tensor categories, Lemma~\ref{lemm:negligible}
says that $\conncat$ has no nonzero negligible morphisms; it is the
single point at which the definition of $\conncat$ as a quotient
\emph{by} the connection kernel pays off, and it is the
engine of semisimplicity in Section~\ref{sec:semisimple}.

\subsection{The Cauchy completion}\label{subsec:cauchy}

The category $\conncat$ need not admit all finite direct sums or
split all its idempotents, and both are needed below. We
therefore pass to the
\emph{Cauchy completion}
\[
\cauchy\coloneqq\Kar\bigl(\Add(\conncat)\bigr),
\]
where $\Add$ is the additive envelope --- objects are finite tuples
of objects of $\conncat$, morphisms are matrices --- and $\Kar$ is
the Karoubi envelope --- objects are pairs $(Y,e)$ with
$e\in\End(Y)$ idempotent, and morphisms $(Y,e)\to(Z,e')$ are the
$a\colon Y\to Z$ with $a=e'ae$. The Karoubi envelope of an additive
category is additive, with $(Y,e)\oplus(Z,e')=(Y\oplus Z,e\oplus
e')$, so $\cauchy$ admits finite direct sums and split idempotents
together.\footnote{The order matters: applying $\Add$ \emph{after}
$\Kar$ need not produce an idempotent-complete category, since a
summand of a direct sum need not be a direct sum of summands. The
distinction disappears once $\cauchy$ is known to be semisimple
(Section~\ref{sec:semisimple}), after which
$\Add(\Kar(\conncat))\simeq\cauchy$.}

The structure of $\conncat$ extends. The tensor product extends
bilinearly to tuples and to idempotent pairs, with the symmetry
acting componentwise; morphism spaces remain finite-dimensional,
being subspaces of finite direct sums of morphism spaces of
$\conncat$; the canonical inclusion is fully faithful, so
$\End(\unitobj)=\mathbb{C}$ is unchanged; and duals extend by
$(Y_{1},\dotsc,Y_{m})^{\vee}=(Y_{1}^{\vee},\dotsc,Y_{m}^{\vee})$
and $(Y,e)^{\vee}=(Y^{\vee},e^{\vee})$, with evaluation and
coevaluation corrected by the idempotents --- the cup is
$(e\otimes e^{\vee})\circ\mathrm{coev}_{Y}$ and the cap
$\mathrm{ev}_{Y}\circ(e^{\vee}\otimes e)$; the first snake
identity for $(Y,e)$ then unwinds to the defining formula of the
dual morphism $e^{\vee}$, and the second collapses onto the snake for
$Y$, and on tuples the snakes reduce entrywise. Traces are computed accordingly:
the trace of a matrix of endomorphisms is the sum of the traces of
its diagonal entries, and the trace of $a\in\End(Y,e)$ is
$\tr_{Y}(a)$. Thus $\cauchy$ is a rigid symmetric
$\mathbb{C}$-linear monoidal category with finite-dimensional
morphism spaces and $\End(\unitobj)=\mathbb{C}$, and it is
generated by $\genX$
under tensor products, direct sums and summands.

The nondegeneracy of Lemma~\ref{lemm:negligible} survives the
completion.

\begin{lemm}\label{lemm:nondegD}
Let $P$ and $Q$ be objects of $\cauchy$ and let $a\colon P\to Q$
satisfy $\tr(ba)=0$ for every $b\colon Q\to P$. Then $a=0$.
\end{lemm}

\begin{proof}
First let $P$ and $Q$ be objects of $\Add(\conncat)$, say with
components $(Y_{i})_{i}$ and $(Z_{j})_{j}$, and let
$a=(a_{ji})$ as a matrix. For fixed indices $i$ and $j$, let $b$
be the matrix supported in entry $(i,j)$ with value an arbitrary
$b_{ij}\colon Z_{j}\to Y_{i}$. Then
$\tr(ba)=\tr(b_{ij}a_{ji})=0$ for every $b_{ij}$, and
Lemma~\ref{lemm:negligible} gives $a_{ji}=0$; as $i$ and $j$ were
arbitrary, $a=0$.

Now let $P=(Y,e)$ and $Q=(Z,e')$ be objects of $\cauchy$ and
$a=e'ae\colon P\to Q$ with $\tr(ba)=0$ for all $b\colon Q\to P$.
For an arbitrary $c\colon Z\to Y$ in $\Add(\conncat)$, put
$b\coloneqq ece'$, a morphism $Q\to P$. Then, using $a=e'ae$ and
Lemma~\ref{lemm:tracecalc}(a),
\[
\tr(ca)=\tr(c\,e'ae)=\tr(ece'\,a)=\tr(ba)=0 ,
\]
and the first paragraph gives $a=0$.
\end{proof}

\section{Moderate growth, nilpotent traces, and the super fibre
functor}\label{sec:semisimple}

We are now ready to prove that $\cauchy$ is semisimple abelian
and admits a super fibre functor. The whole
argument rests on three facts: the endomorphism spaces of tensor
powers in $\cauchy$ grow at most exponentially; nilpotent
endomorphisms have trace zero;
and the trace pairing is nondegenerate
(Lemma~\ref{lemm:nondegD}). The rest is finite-dimensional algebra
and a theorem of Deligne.

\subsection{Exponential growth and nilpotent traces}
\label{subsec:growth}

\begin{lemm}\label{lemm:growthD}
Let $Z\in\cauchy$ be a direct summand of
$\bigoplus_{j=1}^{m}\genX^{\otimes n_{j}}$ with $m\ge1$, and put
$n_{\max}\coloneqq\max_{j}n_{j}$. Then
\[
\dim\End(Z^{\otimes N})\le m^{2N}R^{2n_{\max}N}
\qquad(N\ge 0).
\]
In particular $\dim\End(Z^{\otimes N})<N!$ for all sufficiently
large $N$. Every object of $\cauchy$ is such a summand.
\end{lemm}

\begin{proof}
Put $Y\coloneqq\bigoplus_{j=1}^{m}\genX^{\otimes n_{j}}$ and
choose $\iota\colon Z\to Y$ and $\rho\colon Y\to Z$ with
$\rho\iota=\mathrm{id}_{Z}$. Tensoring this retraction shows that
$Z^{\otimes N}$ is a retract of $Y^{\otimes N}$, and
$g\mapsto\iota^{\otimes N}g\rho^{\otimes N}$ embeds
$\End(Z^{\otimes N})$ linearly into $\End(Y^{\otimes N})$, with
left inverse $a\mapsto\rho^{\otimes N}a\iota^{\otimes N}$. The
$N$-th tensor power of the sum expands into $m^{N}$ summands
$\genX^{\otimes|w|}$ indexed by words $w$, each with
$|w|\le n_{\max}N$, so
\[
\dim\End(Z^{\otimes N})
\le\sum_{w,w'}\dim\Hom\bigl(\genX^{\otimes|w|},
\genX^{\otimes|w'|}\bigr)
\le m^{2N}\,R^{2n_{\max}N},
\]
using (H2) for each summand. The right-hand side is exponential in
$N$, hence eventually smaller than $N!$.
\end{proof}

\begin{lemm}[Cycle-trace identity]\label{lemm:cycletrace}
Let $Y\in\cauchy$, let $y_{1},\dotsc,y_{m}\in\End(Y)$ with
$m\ge1$, and let $\Phi_{m}(\pi)\in\End(Y^{\otimes m})$ denote the
action of $\pi\in S_{m}$ through the symmetry, which carries the
$i$-th tensor factor to the $\pi(i)$-th position. Then
\[
\tr\bigl(\Phi_{m}(\pi)\circ(y_{1}\otimes\dotsb\otimes y_{m})\bigr)
=\prod_{c}\tr\bigl(y_{i_{r}}\dotsm y_{i_{2}}y_{i_{1}}\bigr),
\]
the product running over the cycles
$c=(i_{1}\,i_{2}\,\dotsc\,i_{r})$ of $\pi$, fixed points included,
the factors of each cycle being composed in the order in which the
cycle visits them. In particular
$\tr(\Phi_{m}(\pi)\circ y^{\otimes m})=\prod_{c}\tr(y^{|c|})$ for
every $y\in\End(Y)$.
\end{lemm}

\begin{proof}
Let first $Y=\genX^{\otimes t}$ and let each $y_{i}$ be the class
of a fragment. Both sides are then $f$-values of one closed graph:
the strand closure of
$\Phi_{m}(\pi)\circ(y_{1}\otimes\dotsb\otimes y_{m})$ joins the
outputs of the copy $y_{i}$ to the inputs of $y_{\pi(i)}$, so that
the copies belonging to a cycle $c$ form one closed necklace, which
is the closure of the composite $y_{i_{r}}\dotsm y_{i_{1}}$; the
necklaces of distinct cycles are disjoint, and $f$ multiplies over
them by Lemma~\ref{lemm:mult}. Both sides are multilinear in
$(y_{1},\dotsc,y_{m})$, so the identity extends from fragments to
all $y_{i}\in\End(\genX^{\otimes t})$. It passes to
$\Add(\conncat)$ by the entrywise computation of traces and of
matrix products, and to a Karoubi object $(Y,e)$ because
$\tr_{(Y,e)}$ is $\tr_{Y}$ on morphisms of the form $eye$. The
last assertion is the case $y_{1}=\dotsb=y_{m}=y$.
\end{proof}

\begin{theo}\label{theo:niltrace}
Every nilpotent endomorphism in $\cauchy$ has trace zero.
\end{theo}

\begin{proof}
This is Schrijver's argument \cite[Proposition~4]{Sch15},
transferred from his algebras of fragments to $\cauchy$. Let
$g\in\End(Y)$ be nilpotent with $\tr(g)\ne0$. Since $g^{M}=0$ for
some $M$, there is a largest exponent $t\ge1$ with
$\tr(g^{t})\ne0$; put $y\coloneqq g^{t}$, so that $\tr(y)\ne0$ and
$\tr(y^{s})=0$ for every $s\ge2$. For $m\ge1$ and $\pi\in S_{m}$,
Lemma~\ref{lemm:cycletrace} gives
\[
\tr\bigl(\Phi_{m}(\pi)\circ y^{\otimes m}\bigr)=
\begin{cases}
\tr(y)^{m}&\pi=\mathrm{id},\\
0&\text{otherwise,}
\end{cases}
\]
since a permutation other than the identity has a cycle of length
at least two. Hence the $m!$ endomorphisms $\Phi_{m}(\pi)$ of
$Y^{\otimes m}$ are linearly independent: if
$\sum_{\pi}\lambda_{\pi}\Phi_{m}(\pi)=0$, then composing with
$\Phi_{m}(\rho^{-1})$ and applying the functional
$a\mapsto\tr(a\circ y^{\otimes m})$ leaves
$\lambda_{\rho}\tr(y)^{m}=0$, so $\lambda_{\rho}=0$ for every
$\rho$. Thus $m!\le\dim\End(Y^{\otimes m})$ for every $m\ge1$,
contradicting Lemma~\ref{lemm:growthD} for large $m$.
\end{proof}

Etingof and Penneys have proved the vanishing of nilpotent traces
in much greater generality, for braided categories of moderate
growth \cite[Cor.~1.5, Lem.~1.6]{EP26}, a notion which the
exponential bound of Lemma~\ref{lemm:growthD} supplies; without a
growth hypothesis there are rigid symmetric categories with
nilpotent endomorphisms of nonzero trace \cite[\S3.4]{EP26}.
Appendix~\ref{app:zeta} proves more in our setting: applying
Corollary~\ref{coro:zetagrowth} with $A=m^{2}R^{2n_{\max}}$ makes
the trace zeta function of every endomorphism rational, with
explicit degree bounds.

\subsection{Semisimplicity}\label{subsec:semisimplicity}

\begin{theo}\label{theo:radical}
For every $Y\in\cauchy$ the algebra $\End(Y)$ is semisimple.
\end{theo}

\begin{proof}
Let $j$ be an element of the Jacobson radical of the
finite-dimensional algebra $\End(Y)$ and let $u\in\End(Y)$ be
arbitrary. Then $ju$ lies in the radical, hence is nilpotent ---
the radical of a finite-dimensional algebra is a nilpotent ideal
--- so $\tr(ju)=0$ by Theorem~\ref{theo:niltrace}, and by cyclicity
(Lemma~\ref{lemm:tracecalc}(a)) also $\tr(uj)=0$. Now
Lemma~\ref{lemm:nondegD} applied with $a=j$ gives $j=0$.
\end{proof}

Decompose the identity of any $Y\in\cauchy$ into orthogonal
primitive idempotents of the semisimple algebra $\End(Y)$; since
idempotents split in $\cauchy$, this expresses $Y$ as a finite
direct sum of objects with primitive identity. We call an object
$S$ with $\End(S)=\mathbb{C}$ an \emph{atom} --- the word
\emph{simple} is reserved until the abelian structure is
available --- and the summands just
produced are atoms, because a corner $e\End(Y)e$ of a semisimple
$\mathbb{C}$-algebra at a primitive idempotent is a
finite-dimensional division algebra over the algebraically closed
field $\mathbb{C}$, hence $\mathbb{C}$ itself.

\begin{lemm}\label{lemm:dichotomy}
A nonzero morphism between atoms is an isomorphism, and
non-isomorphic atoms admit no nonzero morphisms between them.
\end{lemm}

\begin{proof}
Let $a\colon S\to T$ be nonzero with $S$, $T$ atoms. By
Lemma~\ref{lemm:nondegD} there is $b\colon T\to S$ with
$\tr(ba)\ne0$. Since $\End(S)=\mathbb{C}$ we have
$ba=\mu\,\mathrm{id}_{S}$ with $\mu\ne0$, so
$v\coloneqq\mu^{-1}b$ satisfies $va=\mathrm{id}_{S}$; then
$av\in\End(T)=\mathbb{C}$ is an idempotent, nonzero because
$ava=a\ne0$, hence $av=\mathrm{id}_{T}$. Thus $a$ is an
isomorphism.
\end{proof}

\begin{theo}\label{theo:semisimple}
Choose representatives $\{S_{\gamma}\}$ of the isomorphism classes
of atoms. The functor
$Y\mapsto\bigl(\Hom(S_{\gamma},Y)\bigr)_{\gamma}$ is an
equivalence of $\mathbb{C}$-linear categories between $\cauchy$
and the category of finite-support families of finite-dimensional
vector spaces. Consequently $\cauchy$ is semisimple abelian, its
tensor product is exact in each variable, and its simple objects
are the atoms.
\end{theo}

\begin{proof}
Every object is a finite direct sum of atoms, as shown above. By
Lemma~\ref{lemm:dichotomy}, for $Y\cong\bigoplus S_{\gamma(i)}$
and $Z\cong\bigoplus S_{\delta(j)}$ every morphism is a matrix
over the spaces
$\Hom(S_{\gamma},S_{\delta})=\delta_{\gamma\delta}\cdot\mathbb{C}$,
so the functor is fully faithful; a finite-support family
$(W_{\gamma})$ is the image of
$\bigoplus_{\gamma}S_{\gamma}^{\oplus\dim W_{\gamma}}$, so it is
essentially surjective. The target category is semisimple abelian,
and abelianness transports along the equivalence. Exactness of the
tensor product holds because in a semisimple abelian category
every short exact sequence splits and additive functors preserve
split exactness.
\end{proof}

\subsection{The super fibre functor}\label{subsec:fibrefunctor}

We may now verify the hypotheses of Deligne's theorem for
$\cauchy$:
\begin{itemize}
\item essentially small: the fragments form a set of
  representatives, and all Hom-spaces, idempotents, tuples and
  pairs built from them are set-sized;
\item abelian, $\mathbb{C}$-linear and semisimple, with
  finite-dimensional morphism spaces and with
  $\End(\unitobj)=\mathbb{C}$: Theorem~\ref{theo:semisimple} and
  Section~\ref{subsec:cauchy};
\item rigid and symmetric, with exact tensor product:
  Section~\ref{subsec:cauchy} and
  Theorem~\ref{theo:semisimple};
\item finitely tensor-generated by the self-dual object $\genX$:
  every object of $\cauchy$ is a direct summand of a finite
  direct sum of tensor powers of $\genX$; summands are subobjects
  in the semisimple abelian structure, and $\genX$ is self-dual,
  so this is tensor generation in Deligne's sense;
\item every object of moderate growth: in a semisimple category
  the length of $Y^{\otimes N}$ is at most
  $\dim\End(Y^{\otimes N})$ --- if
  $Y^{\otimes N}\cong\bigoplus_{i}S_{i}^{\oplus m_{i}}$ then the
  length is $\sum_{i}m_{i}\le\sum_{i}m_{i}^{2}=\dim\End$ --- and
  the latter is exponentially bounded by
  Lemma~\ref{lemm:growthD}.
\end{itemize}
In this list the nilpotent-trace theorem enters only through
Theorem~\ref{theo:semisimple}, and everything else follows from
the rank bound and the constructions of
Section~\ref{sec:conncat}.

\begin{theo}[{Deligne \cite[Thm~0.6]{Del02}; see also
\cite{Ost04}}]\label{theo:fibre}
There is an exact faithful $\mathbb{C}$-linear symmetric monoidal
functor $\fib\colon\cauchy\to\svect$.
\end{theo}

Deligne's theorem concludes more than is stated here: a tensor
equivalence with a category $\operatorname{Rep}(G,\varepsilon)$
of super representations of an affine supergroup scheme
\cite[0.3]{Del02}, of which the functor above is the composite
with the forgetful functor. Of the functor's properties,
Sections~\ref{sec:extraction} and~\ref{sec:circuit} use that it
is symmetric monoidal and $\mathbb{C}$-linear, the dimension
bound below uses its linearity, and
Section~\ref{sec:consequences} uses its faithfulness, to identify
the symmetric-group blocks surviving in $\cauchy$ with those
acting on the model.

Here the morphisms of $\svect$ are the parity-preserving linear
maps; we use this convention repeatedly. We fix such an $\fib$ once
and for all, suppress its structural isomorphisms
$\fib(Y)\otimes\fib(Z)\cong\fib(Y\otimes Z)$ from the notation
throughout, and write $V\coloneqq\fib(\genX)$, a super vector
space
$V=V_{0}\oplus V_{1}$ with $\dim V_{0}=k$ and, as we shall see in
Section~\ref{sec:extraction}, $\dim V_{1}=2\ell$ even.

The growth hypothesis also bounds the dimension of $V$; this is
the source of the quantitative clause of the Main Theorem. Here
$\Phi_{n}$ denotes the symmetric-braiding action of
$\mathbb{C}[S_{n}]$ on the $n$-th tensor power of an object,
$e_{\lambda}$ the central idempotent of the block of
$\lambda\vdash n$, and $d_{\lambda}$ the dimension of the
irreducible $S_{n}$-module $U_{\lambda}$; as in
Appendix~\ref{app:zeta}, we say that $\lambda$ \emph{survives}
for an object $Z$ if $\Phi_{n}(e_{\lambda})\ne0$ in
$\End(Z^{\otimes n})$, and that it \emph{dies} otherwise. We write
$\Phi^{V}_{n}\coloneqq\fib\circ\Phi_{n}$ for the resulting action
of $\mathbb{C}[S_{n}]$ on $V^{\otimes n}$; since $\fib$ is
symmetric monoidal, this is the super-permutation action, in which
a permutation rearranges the tensor factors with the Koszul sign of
the rearrangement.

\begin{lemm}[Block faithfulness]\label{lemm:blockimport}
The kernel of $\Phi_{n}$ is the sum of the blocks of the
partitions of $n$ that die for $\genX$; consequently
\[
\sum_{\lambda\text{ survives}}d_{\lambda}^{2}
\le\dim\End(\genX^{\otimes n})=\rk\connmat{f}{2n}.
\]
\end{lemm}

\begin{proof}
The kernel of the algebra homomorphism
$\Phi_{n}\colon\mathbb{C}[S_{n}]\to\End(\genX^{\otimes n})$ is a
two-sided ideal of the semisimple algebra
$\mathbb{C}[S_{n}]=\bigoplus_{\lambda\vdash n}B_{\lambda}$, where
$B_{\lambda}\cong\mathrm{Mat}_{d_{\lambda}}(\mathbb{C})$ is the
block of $\lambda$; hence the kernel is a sum of blocks, and
$B_{\lambda}$ lies in it iff $\Phi_{n}(e_{\lambda})=0$, since
$x=e_{\lambda}x$ for $x\in B_{\lambda}$. So $\Phi_{n}$ is
injective on the sum of the surviving blocks, whose dimension is
the left-hand side, and $\dim\End(\genX^{\otimes n})
=\rk\connmat{f}{2n}$ by Section~\ref{subsec:categorydef}.
\end{proof}

For $\lambda\vdash n$ let
$\schur{\lambda}{V}\coloneqq\Hom_{S_{n}}(U_{\lambda},V^{\otimes n})$
be the multiplicity space of $U_{\lambda}$ in $V^{\otimes n}$ under
$\Phi^{V}_{n}$, where $\Hom_{S_{n}}$ denotes the space of all
intertwining linear maps, with its natural super grading, so that
its dimension is the total dimension of its even and odd parts;
the modules $U_{\lambda}$ being self-dual, this is the super Schur
functor of Deligne \cite[\S1]{Del02} applied to $V$. Then
\begin{equation}\label{eq:isotypic}
V^{\otimes n}\cong\bigoplus_{\lambda\vdash n}
U_{\lambda}\otimes\schur{\lambda}{V},
\end{equation}
and $\Phi^{V}_{n}(e_{\lambda})$ is the projection onto the
$\lambda$-th summand; in particular $\schur{\lambda}{V}\ne0$ iff
$\Phi^{V}_{n}(e_{\lambda})\ne0$. Since $\fib$ is linear, if
$\Phi_{n}(e_{\lambda})=0$ then
$\Phi^{V}_{n}(e_{\lambda})=\fib(\Phi_{n}(e_{\lambda}))=0$; so if
$\schur{\lambda}{V}\ne0$ then $\lambda$ survives for $\genX$. The
converse, which needs faithfulness, is not used until
Section~\ref{sec:consequences}.

\begin{lemm}[The dimension inequality]\label{lemm:diminequality}
Put $d\coloneqq\dim V=\dim V_{0}+\dim V_{1}$ and suppose $d\ge1$.
Then for every $n\ge1$,
\[
d^{2n}\le\rk\connmat{f}{2n}\cdot\binom{n+d^{2}-1}{d^{2}-1}.
\]
\end{lemm}

\begin{proof}
Fix $n\ge1$ and write $m_{\lambda}\coloneqq\dim\schur{\lambda}{V}$.
Taking dimensions in \eqref{eq:isotypic}, with the grading
forgotten,
\[
d^{n}=\sum_{\lambda\vdash n}d_{\lambda}m_{\lambda}.
\]
Every $\lambda$ with $m_{\lambda}\ne0$ survives for $\genX$, as
just noted, so Lemma~\ref{lemm:blockimport} gives
$\sum_{m_{\lambda}\ne0}d_{\lambda}^{2}\le\rk\connmat{f}{2n}$;
and by Schur's lemma
$\sum_{\lambda}m_{\lambda}^{2}=\dim\End_{S_{n}}(V^{\otimes n})$,
where $\End_{S_{n}}(V^{\otimes n})$ consists of all linear
endomorphisms of $V^{\otimes n}$ commuting with the action,
parity-reversing ones included. The Cauchy--Schwarz inequality
therefore gives
\[
d^{2n}=\Bigl(\sum_{m_{\lambda}\ne0}d_{\lambda}m_{\lambda}\Bigr)^{2}
\le\Bigl(\sum_{m_{\lambda}\ne0}d_{\lambda}^{2}\Bigr)
\Bigl(\sum_{\lambda}m_{\lambda}^{2}\Bigr)
\le\rk\connmat{f}{2n}\cdot\dim\End_{S_{n}}(V^{\otimes n}),
\]
and it remains to bound the commutant by counting orbits. Choose
a homogeneous basis $v_{1},\dotsc,v_{d}$ of $V$; the tensors
$v_{w}\coloneqq v_{w(1)}\otimes\dotsb\otimes v_{w(n)}$, indexed by
the words $w\in[d]^{n}$, form a basis of $V^{\otimes n}$ on which
the action is monomial: $\Phi^{V}_{n}(\sigma)$ sends $v_{w}$ to
$\pm v_{w\circ\sigma^{-1}}$. Let $T\in\End_{S_{n}}(V^{\otimes n})$
have matrix $(T_{w,w'})$ in this basis. Comparing matrix entries
in $T\Phi^{V}_{n}(\sigma)=\Phi^{V}_{n}(\sigma)T$ equates
$T_{w\circ\sigma,w'\circ\sigma}$ with $\pm T_{w,w'}$ for every
$\sigma\in S_{n}$, so $T$ is determined by its entries at one
representative of each orbit of $S_{n}$ acting on pairs of words
by simultaneous permutation of positions. Two such pairs lie in
the same orbit iff for every pair of letters $(a,b)\in[d]^{2}$
they have the same number of positions $i$ carrying $(a,b)$,
i.e.\ with $(w(i),w'(i))=(a,b)$: a simultaneous permutation of
positions preserves the counts, and when the counts agree, a
permutation matching the positions letter pair by letter pair
carries one pair of words to the other. An orbit is thus
determined by its $d^{2}$ letter-pair counts, nonnegative
integers summing to $n$, and there are
$\binom{n+d^{2}-1}{d^{2}-1}$ such tuples; hence
$\dim\End_{S_{n}}(V^{\otimes n})\le\binom{n+d^{2}-1}{d^{2}-1}$.
\end{proof}

Read as a lower bound,
$\rk\connmat{f}{2n}\ge d^{2n}\big/\tbinom{n+d^{2}-1}{d^{2}-1}$
for every single $n$; the exponential against the polynomial is
what the corollary exploits.

\begin{coro}\label{coro:dimbound}
The dimensions of $V$ satisfy $\dim V_{0}+\dim V_{1}\le R$.
\end{coro}

\begin{proof}
Let $d\coloneqq\dim V$; if $d=0$ there is nothing to prove.
Otherwise (H2) gives $\rk\connmat{f}{2n}\le R^{2n}$, so
Lemma~\ref{lemm:diminequality} yields
$d\le R\bigl(\tbinom{n+d^{2}-1}{d^{2}-1}\bigr)^{1/2n}$ for every
$n\ge1$; since $\tbinom{n+d^{2}-1}{d^{2}-1}\le(n+1)^{d^{2}-1}$, the
root tends
to $1$ as $n\to\infty$, and $d\le R$.
\end{proof}

\section{From categorical tensors to a super tensor
network}\label{sec:extraction}

In this section we extract from the fibre functor the raw data
of a mixed model --- a bilinear form, a copairing, and one vertex
tensor per degree --- and express $f$ as the value of a super tensor
network built from them. The data map is
\[
\mathrm{ev}\mapsto b,\qquad
\mathrm{coev}\mapsto\copair,\qquad
\sigma_{d}\mapsto T_{d}\rightsquigarrow h^{\mathrm{cat}}_{d},\qquad
(\copair,h^{\mathrm{cat}})\rightsquigarrow
(\twistedC,h^{\mathrm{RS}})
\rightsquigarrow\text{Definition~\ref{defi:mpf}},
\]
the identification with Definition~\ref{defi:mpf} being the business of
Section~\ref{sec:circuit}. The objects introduced along the way,
and their relations, are collected in the following table.

\begin{center}
\begin{tabular}{@{}lp{8.8cm}l@{}}
$b=\fib(\mathrm{ev})$ & the form: even, supersymmetric,
  nondegenerate & \S\ref{subsec:formcoords}\\
$\copair=\fib(\mathrm{coev})$ & the copairing:
  \mbox{$L_{\copair}=\mathrm{id}_{V}$},
  \mbox{$b(\copair)=k-2\ell=f(\freecircle)$} & \S\ref{subsec:copairing}\\
$\twistedC=(\parityop\otimes\mathrm{id})\copair$ & the twisted element
  of Definition~\ref{defi:mpf}: \mbox{$L_{\twistedC}=\parityop$},
  \mbox{$b(\twistedC)=k+2\ell$} & \S\ref{subsec:copairing}\\
$T_{d}=\fib(\sigma_{d})$ & the vertex tensors: even and
  $S_{d}$-invariant & \S\ref{subsec:mates}\\
$h^{\mathrm{cat}}$ & the functional matching $\copair$:
  $\beta_{d}(T_{d},-)=h^{\mathrm{cat}}_{d}\circ\symq{d}$ &
  \S\ref{subsec:mates}\\
$\parinv$ & the scalar $(-1)^{j(j-1)/2}$ on exterior degree $j$ &
  \S\ref{subsec:paritytransfer}\\
$h^{\mathrm{RS}}=h^{\mathrm{cat}}\circ\parinv$ & the functional
  matching $\twistedC$, viz.\ that of Definition~\ref{defi:mpf} &
  \S\ref{subsec:paritytransfer}
\end{tabular}
\end{center}

\subsection{The form and its coordinates}\label{subsec:formcoords}

Write $b\coloneqq\fib(\mathrm{ev})\colon V\otimes V\to\mathbb{C}$.
The form $b$ is even, because the morphisms of $\svect$ are
parity-preserving;
supersymmetric, since $\mathrm{ev}\circ c=\mathrm{ev}$ holds
already in $\conncat$ (gluing the swap into a strand yields a
strand); and nondegenerate, since the snake identities make
$x\mapsto b(x,-)$ invertible with inverse built from
$\fib(\mathrm{coev})$. The restriction of $b$ to $V_{1}$ is then a
nondegenerate antisymmetric form, so $\dim V_{1}$ is even; we
write $\dim V_{1}=2\ell$, and Corollary~\ref{coro:dimbound} reads
$k+2\ell\le R$. Choose
coordinates
adapted to $b$: a basis $e_{1},\dotsc,e_{k}$ of $V_{0}$ with
$b(e_{i},e_{j})=\delta_{ij}$ and $b(e_{i},V_{1})=0$, and a basis
$\xi_{1},\dotsc,\xi_{2\ell}$ of $V_{1}$ with
\[
b(\xi_{m},\xi_{m+\ell})=1=-b(\xi_{m+\ell},\xi_{m})
\quad(m\le\ell),
\qquad
b(\xi_{i},\xi_{j})=0\ \text{otherwise}.
\]
With $\eta_{i}$ as in Section~\ref{subsec:mpf}, this gives
\begin{equation}\label{eq:etapairing}
b(\eta_{i},\xi_{j})=\delta_{ij}
\qquad\text{and}\qquad
b(\xi_{i},\eta_{i})=-1
\qquad\text{for all }i,j\text{:}
\end{equation}
for $i\le\ell$,
$b(\eta_{i},\xi_{j})=-b(\xi_{i+\ell},\xi_{j})=-(-\delta_{ij})
=\delta_{ij}$, and for $i>\ell$,
$b(\eta_{i},\xi_{j})=b(\xi_{i-\ell},\xi_{j})=\delta_{ij}$, using
$b(\xi_{m},\xi_{m+\ell})=1$; the second identity follows by
supersymmetry.

For homogeneous $v_{i},w_{i}\in V$ define the
\emph{tensor-power pairing}
\[
\beta_{d}(v_{1}\otimes\dotsb\otimes v_{d},\,
w_{1}\otimes\dotsb\otimes w_{d})
\coloneqq
(-1)^{\sum_{i<j}|v_{j}||w_{i}|}\prod_{i=1}^{d}b(v_{i},w_{i}).
\]

\begin{lemm}\label{lemm:koszul}
The Koszul sign of a rearrangement of a word of homogeneous
tensor factors in $\svect$ equals the sign of the permutation it
induces on the odd letters; in particular the signs multiply
under successive rearrangements, the parity word being
transported after each, and signed evaluations along any two
rearrangement paths realising the same permutation of the
labelled homogeneous tensor factors are equal. Consequently:
\begin{enumerate}
\item[(a)] $\beta_{d}$ is the signed evaluation obtained by
  rearranging $(v_{1}\dotsm v_{d}\,w_{1}\dotsm w_{d})$ into the
  interleaved word $(v_{1}w_{1}\,v_{2}w_{2}\dotsm v_{d}w_{d})$
  and applying $b$ to each adjacent pair;
\item[(b)] $\beta_{d}(\pi v,\pi w)=\beta_{d}(v,w)$ for the super
  $S_{d}$-action applied to both arguments.
\end{enumerate}
\end{lemm}

\begin{proof}
The symmetry of $\svect$ on homogeneous vectors is the flip times
$(-1)^{|x||y|}$, so a composite of symmetries multiplies the
scalars $(-1)$ once for each transposition of two odd letters;
this is precisely the sign of the induced permutation of the odd
letters, and induced permutations compose along successive
rearrangements, so the signs multiply; any two decompositions
realising the same permutation of the labelled homogeneous
tensor factors agree by coherence of the symmetric structure. For (a), moving each $w_{i}$ leftward past
$v_{i+1},\dotsc,v_{d}$ costs exactly
$(-1)^{\sum_{i<j}|v_{j}||w_{i}|}$ on the support of the product of
the $b(v_{i},w_{i})$. For (b), each factor $b(v_{i},w_{i})$
forces $|v_{i}|=|w_{i}|$ on the support, so every interleaved
block $(v_{i}w_{i})$ has even total parity; acting by $\pi$ on
both arguments permutes these even blocks, which costs no sign,
and merely reindexes the scalar factors. Terms outside the
support vanish on both sides.
\end{proof}

\subsection{The copairing and the twisted
element}\label{subsec:copairing}

Write $\copair\coloneqq\fib(\mathrm{coev})\in V\otimes V$, and for
$D\in V\otimes V$ define $L_{D}\colon V\to V$ by
$L_{D}(x)\coloneqq(\mathrm{id}\otimes b)(D\otimes x)$, with no
reordering, so that $L_{D}(x)=\sum_{i}u_{i}\,b(w_{i},x)$ for
$D=\sum_{i}u_{i}\otimes w_{i}$. Applying $\fib$ to the snake
identity, which holds strictly in $\conncat$ by strand
unification, gives $L_{\copair}=\mathrm{id}_{V}$; and since $b$ is
nondegenerate, $L_{D}=L_{D'}$ forces $D=D'$, so this property
determines $\copair$ uniquely. In the coordinates above,
\[
\copair=\copair_{0}+\copair_{1},
\qquad
\copair_{0}=\sum_{i=1}^{k}e_{i}\otimes e_{i},
\qquad
\copair_{1}=\sum_{i=1}^{2\ell}\xi_{i}\otimes\eta_{i}:
\]
indeed $L_{\copair_{0}}(e_{j})=e_{j}$ directly,
$L_{\copair_{1}}(\xi_{j})=\sum_{i}\xi_{i}\,b(\eta_{i},\xi_{j})
=\xi_{j}$ by \eqref{eq:etapairing}, and the cross-parity terms
vanish because $b$ is even, so
$L_{\copair_{0}+\copair_{1}}=\mathrm{id}_{V}$. Unwinding the
definition of the $\eta_{i}$ gives the alternative form
\begin{equation}\label{eq:oddcopair}
\sum_{i}\eta_{i}\otimes\xi_{i}
=\sum_{m\le\ell}\bigl(-\xi_{m+\ell}\otimes\xi_{m}
+\xi_{m}\otimes\xi_{m+\ell}\bigr)
=-\copair_{1}.
\end{equation}

Let $\parityop\colon V\to V$ denote the parity operator, and
define the \emph{twisted element}
\[
\twistedC\coloneqq(\parityop\otimes\mathrm{id})\copair
=\copair_{0}-\copair_{1},
\qquad\text{so that}\qquad
L_{\twistedC}=\parityop ,
\]
and, by \eqref{eq:oddcopair}, its odd part is
\[
\twistedC_{1}=-\sum_{i}\xi_{i}\otimes\eta_{i}
=\sum_{i}\eta_{i}\otimes\xi_{i}.
\]
When $\ell>0$ the twisted element is not a copairing:
\[
b(\copair)=k-2\ell=\operatorname{sdim}V,
\qquad
b(\twistedC)=k+2\ell=\str(\parityop).
\]
Free circles are accordingly carried by multiplicativity and by
$f(\freecircle)=\tr(\mathrm{id}_{\genX})
=\str(\mathrm{id}_{V})=k-2\ell$, the middle equality because a
symmetric monoidal functor takes categorical traces to
supertraces, and they are removed before the parity transfer of
Lemma~\ref{lemm:paritytransfer}.

Both elements are supersymmetric: $c_{V,V}(\copair)=\copair$ and
$c_{V,V}(\twistedC)=\twistedC$, since on the odd part
$c(\xi_{i}\otimes\eta_{i})=-\eta_{i}\otimes\xi_{i}$ while
$\sum_{i}\eta_{i}\otimes\xi_{i}=-\sum_{i}\xi_{i}\otimes\eta_{i}$
by \eqref{eq:oddcopair}, and the two signs cancel. This
is what licenses assigning the two tensor legs of $\copair$, and
likewise of $\twistedC$, to the two half-edges of an edge without
choosing an order.

One further identity involving $\copair$ deserves a name of its
own: it is the mechanism by which inserted copairings will
fuse back into a single one in Lemma~\ref{lemm:indexraising}.

\begin{lemm}[Copairing fusion]\label{lemm:fusion}
Write $\copair=\sum_{i}u_{i}\otimes w_{i}$ with homogeneous terms.
Then
\[
\sum_{i,j}(-1)^{|w_{i}||u_{j}|}\,b(w_{i},w_{j})\;
u_{i}\otimes u_{j}=\copair .
\]
\end{lemm}

\begin{proof}
The left-hand side is a fixed linear contraction applied to
$\copair\otimes\copair$, so it may be computed in the standard
presentation $\copair_{0}+\copair_{1}$, and the even and odd
sectors do not interact since $b$ is even. Even sector: with
$u_{i}=w_{i}=e_{i}$ the sign is trivial and
$\sum_{i,j}b(e_{i},e_{j})\,e_{i}\otimes e_{j}
=\sum_{i}e_{i}\otimes e_{i}=\copair_{0}$. Odd sector: with
$u_{i}=\xi_{i}$, $w_{i}=\eta_{i}$ the sign is $-1$ throughout, and
from the coordinate table of Section~\ref{subsec:formcoords} one
computes $b(\eta_{i},\eta_{j})=\delta_{j,i+\ell}$ for $i\le\ell$
and $b(\eta_{i},\eta_{j})=-\delta_{j,i-\ell}$ for $i>\ell$, so
\[
-\sum_{i,j}b(\eta_{i},\eta_{j})\,\xi_{i}\otimes\xi_{j}
=-\sum_{i\le\ell}\xi_{i}\otimes\xi_{i+\ell}
+\sum_{i>\ell}\xi_{i}\otimes\xi_{i-\ell}
=\sum_{i}\xi_{i}\otimes\eta_{i}=\copair_{1}. \qedhere
\]
\end{proof}

\subsection{Vertex tensors, mates, and index
raising}\label{subsec:mates}

For each $d\ge0$ set
$T_{d}\coloneqq\fib(\sigma_{d})\in V^{\otimes d}$, the
image of the $d$-star of Section~\ref{subsec:categorydef}. A star
has no preferred order on its legs, and this becomes invariance:
for every $\pi\in S_{d}$ we have $P_{\pi}\circ\sigma_{d}
=\sigma_{d}$ in $\conncat$, since gluing strands merely relabels
open ends (Lemma~\ref{lemm:gluingfacts}(ii)) and every
relabelling of the legs of a star is realised by an automorphism
of the underlying fragment; applying $\fib$, which takes
$P_{\pi}$ to the super-permutation action of $\pi$, gives
$\pi\cdot T_{d}=T_{d}$, i.e.\
$T_{d}\in(V^{\otimes d})^{S_{d}}$. Each
$T_{d}$ is moreover \emph{even}, because morphisms from $\unitobj$ are
parity-preserving; this evenness is the ultimate origin of the
Eulerian condition in Definition~\ref{defi:mpf}. Let
$\symq{d}\colon V^{\otimes d}\to\SymV_{s}^{d}V$ be the quotient
onto the super symmetric power --- here and throughout we use the
canonical isomorphism of graded algebras
$\SymV_{s}V\cong\SymV V_{0}\otimes\Lambda V_{1}$, so that the
functionals of Definition~\ref{defi:mpf} act on super symmetric powers --- and
consider the functional $\beta_{d}(T_{d},-)$ on $V^{\otimes d}$.
By Lemma~\ref{lemm:koszul}(b) it is invariant under the super
$S_{d}$-action and therefore factors uniquely as
\[
\beta_{d}(T_{d},-)=h^{\mathrm{cat}}_{d}\circ\symq{d},
\qquad
h^{\mathrm{cat}}_{d}\in\bigl(\SymV_{s}^{d}V\bigr)^{*}.
\]
The functional $h^{\mathrm{cat}}_{d}$ is defined by the
factorisation alone, with no factor of $d!$. Since
$\SymV_{s}V=\bigoplus_{d\ge0}\SymV_{s}^{d}V$, the family
$(h^{\mathrm{cat}}_{d})_{d\ge0}$ assembles into a single
functional $h^{\mathrm{cat}}\in(\SymV_{s}V)^{*}$ --- the algebraic
dual of a direct sum imposes no finiteness condition --- and
likewise $h^{\mathrm{RS}}$ below; it is these assembled
functionals that Definition~\ref{defi:mpf} consumes.

For an even tensor $T\in V^{\otimes d}$, equivalently a morphism
$\mathbb{C}\to V^{\otimes d}$ of $\svect$, define the \emph{mate}
\[
\mate{T}\coloneqq b^{\otimes d}\circ\mathrm{Interleave}_{d}\circ
(T\otimes\mathrm{id}_{V^{\otimes d}})
\colon V^{\otimes d}\to\mathbb{C},
\]
where $\mathrm{Interleave}_{d}$ is the symmetry rearranging
$(v_{1}\dotsm v_{d},w_{1}\dotsm w_{d})\mapsto
(v_{1}w_{1}\dotsm v_{d}w_{d})$ and $b^{\otimes d}$ applies $b$ to
each adjacent pair. By Lemma~\ref{lemm:koszul}(a),
\begin{equation}\label{eq:matecoords}
\mate{T}(w)=\beta_{d}(T,w).
\end{equation}

For a closed graph $G$ without free circles, functionals $H_{v}$
on $V^{\otimes\deg v}$ and an element $D\in V\otimes V$, define
the \emph{tensor network}
\[
\mathrm{TN}'_{(H_{v}),D}(G)\coloneqq
\Bigl(\bigotimes_{v}H_{v}\Bigr)\circ
\mathrm{Regroup}_{G}\Bigl(\bigotimes_{e\in E}D\Bigr),
\]
where $\mathrm{Regroup}_{G}$ is the symmetry taking the word of
$D$-legs in edge order to the word in vertex--half-edge order; it
is a permutation of half-edge--indexed positions, well defined as
a composite of braidings by Lemma~\ref{lemm:koszul}. For a family
$h=(h_{d})$ of functionals on the super symmetric powers we write
$\mathrm{TN}_{h,D}(G)\coloneqq
\mathrm{TN}'_{(h_{\deg v}\circ\symq{\deg v}),D}(G)$; note that
the quotient maps are part of the definition. As written, the
definition presupposes auxiliary choices --- an order of the
edges, an order of the vertices and of the half-edges at each
vertex, and an assignment of the two legs of $D$ to the two
half-edges of each edge --- and the following lemma is what makes
the value a function of the unordered graph.

\begin{lemm}\label{lemm:TNinv}
Suppose that every $H_{v}$ is an even functional invariant under
the super $S_{\deg v}$-action, and that $D$ is even with
$c_{V,V}(D)=D$. Then $\mathrm{TN}'_{(H_{v}),D}(G)$ is independent
of all the auxiliary choices above.
\end{lemm}

\begin{proof}
Two presentations differ by a composite of four moves: a
permutation of the edge factors of $\bigotimes_{e}D$, which costs
no Koszul sign because each factor is an even block
(Lemma~\ref{lemm:koszul}); a swap of the two legs of some factors,
which is absorbed by $c_{V,V}(D)=D$; a permutation of the
half-edge slots at each vertex, which is absorbed by the
invariance of $H_{v}$; and a permutation of the vertex blocks,
which is absorbed by naturality of the symmetry, with no sign
because the $H_{v}$ are even. Koszul functoriality
(Lemma~\ref{lemm:koszul}) then equates the two signed
evaluations.
\end{proof}

Every network appearing below satisfies these hypotheses: the
elements $\copair$ and $\twistedC$ are even and supersymmetric
(Section~\ref{subsec:copairing}); the mates $\mate{T_{d}}$ are
even, since $T_{d}$ and $b$ are, and invariant under the super
$S_{d}$-action by
Lemma~\ref{lemm:koszul}(b) and the invariance of $T_{d}$; and the
functionals $h_{d}\circ\symq{d}$ are even whenever $h$ vanishes in
odd exterior degree, and invariant under the super action by
construction.

\begin{lemm}[Index raising]\label{lemm:indexraising}
For every closed graph $G$ without free circles,
\[
f(G)=\mathrm{TN}'_{(\mate{T_{\deg v}}),\copair}(G)
=\mathrm{TN}_{h^{\mathrm{cat}},\copair}(G).
\]
\end{lemm}

\begin{proof}
The second equality is \eqref{eq:matecoords} together with the
factorisation
$\beta_{d}(T_{d},-)=h^{\mathrm{cat}}_{d}\circ\symq{d}$. For
the first, apply $\fib$ to the star decomposition
\eqref{eq:stars}, giving
\begin{equation}\label{eq:starimage}
f(G)=\Bigl(\bigotimes_{e}b\Bigr)\circ P_{G}\circ
\Bigl(\bigotimes_{v}T_{\deg v}\Bigr),
\end{equation}
with $P_{G}$ now denoting the image of the matching symmetry.
Both snake identities hold strictly in $\conncat$ by strand
unification, and we insert on each of the $2|E|$ wires of
\eqref{eq:starimage} the form
\[
\mathrm{id}_{V}
=(b\otimes\mathrm{id}_{V})\circ(\mathrm{id}_{V}\otimes\copair),
\qquad\text{i.e.}\qquad
x=\sum_{i}b(x,u_{i})\,w_{i}
\quad\text{for }\copair=\sum_{i}u_{i}\otimes w_{i},
\]
chosen so that the star leg is contracted in first argument
position, matching the argument order of $\beta_{d}$. This is the
snake identity complementary to the one that gave
$L_{\copair}=\mathrm{id}_{V}$ in Section~\ref{subsec:copairing}.
Call
$u_{i}$ the \emph{mate leg} and $w_{i}$ the \emph{edge leg} of the
inserted copairing. Regrouping by naturality and coherence of the
symmetry (Lemma~\ref{lemm:koszul}): at each vertex, the inserted
evaluations pair the star legs of $T_{\deg v}$ against the mate
legs, in the order $b(\text{star leg},u_{i})$, which is by
definition the mate $\mate{T_{\deg v}}$ evaluated on the mate
legs; and at each edge, the original evaluation contracts the two
edge legs arriving from its endpoints, leaving the element
\[
\sum_{i,j}(-1)^{|w_{i}||u_{j}|}\,b(w_{i},w_{j})\;
u_{i}\otimes u_{j}=\copair
\]
of $V\otimes V$ feeding the two mate-leg slots --- the Koszul
sign arises from moving $w_{i}$ past $u_{j}$ during the
regrouping, and the identity is the Copairing fusion lemma
(Lemma~\ref{lemm:fusion}). The resulting
composite is by definition
$\mathrm{TN}'_{(\mate{T_{\deg v}}),\copair}(G)$, well defined on the
unordered graph by Lemma~\ref{lemm:TNinv}. We illustrate the
insertion, the two named legs, and the fusion in
Figure~\ref{fig:snake}.
\end{proof}

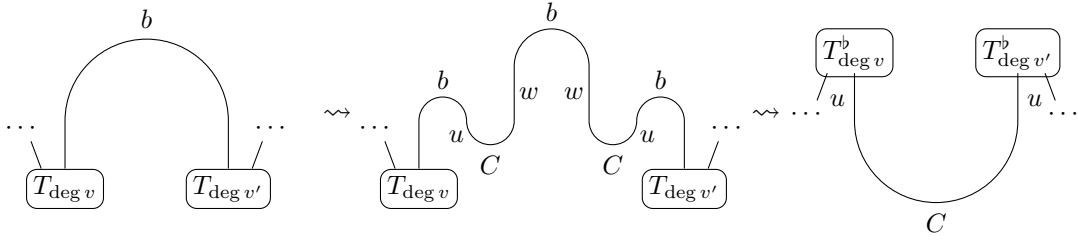
\begin{figure}
\centering
\begin{tikzpicture}[scale=0.9,
  box/.style={draw,rounded corners,inner sep=2.5pt,font=\small},
  lbl/.style={font=\small}]
  \node[box] (TA) at (0,0) {$T_{\deg v}$};
  \node[box] (TB) at (2.4,0) {$T_{\deg v'}$};
  \draw (0,0.3) -- (0,1.0) arc(180:0:1.2) -- (2.4,0.3);
  \node[lbl] at (1.2,2.5) {$b$};
  \draw (-0.35,0.3) -- (-0.5,0.7);
  \draw (2.75,0.3) -- (2.9,0.7);
  \node[lbl] at (-0.62,0.9) {$\cdots$};
  \node[lbl] at (3.05,0.9) {$\cdots$};
  \node at (4.0,1.1) {$\leadsto$};
  \begin{scope}[xshift=5.2cm]
    \node[box] (TC) at (0,0) {$T_{\deg v}$};
    \node[box] (TD) at (3.9,0) {$T_{\deg v'}$};
    \draw (0,0.3) -- (0,1.0) arc(180:0:0.35);         
    \node[lbl] at (0.35,1.6) {$b$};
    \draw (0.7,1.0) arc(180:360:0.35);                 
    \node[lbl] at (1.05,0.35) {$\copair$};
    \node[lbl] at (0.55,0.78) {$u$};
    \node[lbl] at (1.62,1.4) {$w$};
    \draw (1.4,1.0) -- (1.4,1.8) arc(180:0:0.55) -- (2.5,1.0);
    \node[lbl] at (1.95,2.6) {$b$};
    \draw (3.2,1.0) arc(360:180:0.35);                 
    \node[lbl] at (2.85,0.35) {$\copair$};
    \node[lbl] at (2.28,1.4) {$w$};
    \node[lbl] at (3.36,0.78) {$u$};
    \draw (3.9,0.3) -- (3.9,1.0) arc(0:180:0.35);      
    \node[lbl] at (3.55,1.6) {$b$};
    \draw (-0.35,0.3) -- (-0.5,0.7);
    \draw (4.25,0.3) -- (4.4,0.7);
    \node[lbl] at (-0.62,0.9) {$\cdots$};
    \node[lbl] at (4.55,0.9) {$\cdots$};
  \end{scope}
  \node at (10.3,1.1) {$\leadsto$};
  \begin{scope}[xshift=11.6cm]
    \node[box] (TE) at (0,2.0) {$\mate{T_{\deg v}}$};
    \node[box] (TF) at (2.4,2.0) {$\mate{T_{\deg v'}}$};
    \draw (0,1.7) -- (0,1.0) arc(180:360:1.2) -- (2.4,1.7);
    \node[lbl] at (1.2,-0.5) {$\copair$};
    \node[lbl] at (-0.25,1.3) {$u$};
    \node[lbl] at (2.65,1.3) {$u$};
    \draw (-0.4,1.7) -- (-0.55,1.3);
    \draw (2.8,1.7) -- (2.95,1.3);
    \node[lbl] at (-0.68,1.1) {$\cdots$};
    \node[lbl] at (3.1,1.1) {$\cdots$};
  \end{scope}
\end{tikzpicture}
\caption{The snake insertion on one edge. Left: the edge,
evaluated by $b$, between the vertex tensors of its endpoints.
Middle: on each half-edge, the inserted evaluation pairs the star
leg with the mate leg $u$ of a copairing, assembling the mates
$\mate{T}$, while the edge legs $w$ meet the original evaluation.
Right: the edge legs have fused into a single copairing between
the two mates (Lemma~\ref{lemm:fusion}).}
\label{fig:snake}
\end{figure}

\subsection{Parity transfer}\label{subsec:paritytransfer}

The functional $h^{\mathrm{cat}}$ and the copairing $\copair$ do
not yet match the Regts--Sevenster normalisation, which uses the
twisted element. The discrepancy is a sign for each edge of the
odd sector, and it cancels against a compensating sign at the
vertices, as follows. First,
$h^{\mathrm{cat}}_{d}$ vanishes on
$\SymV^{a}V_{0}\otimes\Lambda^{j}V_{1}$ for $j$ odd, since
$T_{d}$ is even and $\beta_{d}$ is an even pairing. Define the
involution $\parinv$ on $\SymV V_{0}\otimes\Lambda V_{1}$ by
$\parinv\coloneqq(-1)^{j(j-1)/2}\cdot\mathrm{id}$ on the summand
of exterior degree $j$ --- for $j=2m$ this is $(-1)^{m}$, and the
odd-$j$ values are immaterial by the vanishing just noted --- and
set
\[
h^{\mathrm{RS}}\coloneqq h^{\mathrm{cat}}\circ\parinv .
\]

\begin{lemm}[Parity transfer]\label{lemm:paritytransfer}
For every closed graph $G$ without free circles,
\[
\mathrm{TN}_{h^{\mathrm{cat}},\copair}(G)
=\mathrm{TN}_{h^{\mathrm{RS}},\twistedC}(G).
\]
\end{lemm}

\begin{proof}
Expand both networks over the set $F$ of edges carrying the odd
part of the element ($\copair_{1}$ on the left, $\twistedC_{1}$
on the right); the even sectors agree term by term. Fix a
surviving term, and let $j_{v}\coloneqq\deg_{F}(v)$ be the number
of odd half-edges at $v$; terms with some $j_{v}$ odd vanish on
both sides by the odd-degree vanishing above, so every $j_{v}$ is
even. Replacing $h^{\mathrm{cat}}$ by $h^{\mathrm{RS}}$
multiplies the term by
$\prod_{v}(-1)^{j_{v}/2}=(-1)^{\frac12\sum_{v}j_{v}}
=(-1)^{|F|}$, by the handshake identity
$\sum_{v}\deg_{F}(v)=2|F|$, which is valid for loops precisely
because a loop contributes two half-edges. Replacing
$\copair_{1}$ by $\twistedC_{1}=-\copair_{1}$ multiplies the term
by $(-1)^{|F|}$ as well. The two factors cancel.
\end{proof}

In coordinates the transfer reads as follows. Call the monomial
\[
m=e_{\mu_{1}}\dotsm e_{\mu_{a}}\otimes
\xi_{i_{1}}\wedge\dotsb\wedge\xi_{i_{j}},
\qquad
\mu_{1}\le\dotsb\le\mu_{a},\quad
i_{1}<\dotsb<i_{j},\quad a+j=d,
\]
\emph{sorted}. The sorted monomials form a basis of
$\SymV V_{0}\otimes\Lambda V_{1}$, and under the canonical
isomorphism of Section~\ref{subsec:mates} the sorted monomial $m$
is the image under $\symq{d}$ of the word
\[
w_{m}\coloneqq e_{\mu_{1}}\otimes\dotsb\otimes e_{\mu_{a}}\otimes
\xi_{i_{1}}\otimes\dotsb\otimes\xi_{i_{j}}.
\]
The tensor words in the letters $e_{1},\dotsc,e_{k}$ and
$\eta_{1},\dotsc,\eta_{2\ell}$ form a basis of $V^{\otimes d}$,
each $\eta_{i}$ being a signed $\xi$; write $T_{d}[u]$ for the
coefficient of $T_{d}$ at the word $u$ of this basis.

\begin{lemm}[Coordinate dictionary]\label{lemm:dictionary}
For every sorted monomial $m$,
\[
h^{\mathrm{RS}}(m)
=T_{d}\bigl[e_{\mu_{1}}\otimes\dotsb\otimes e_{\mu_{a}}\otimes
\eta_{i_{1}}\otimes\dotsb\otimes\eta_{i_{j}}\bigr].
\]
\end{lemm}

\begin{proof}
By definition and by bilinearity,
\[
h^{\mathrm{RS}}(m)=(-1)^{j(j-1)/2}\,\beta_{d}(T_{d},w_{m})
=(-1)^{j(j-1)/2}\sum_{u}T_{d}[u]\,\beta_{d}(u,w_{m}),
\]
the sum running over the words $u$ of the basis. Since $b$ is
even, $\beta_{d}(u,w_{m})$ vanishes unless $u$ carries an even
letter in each of the first $a$ positions and an odd letter in
each of the last $j$, say
$u=e_{\nu_{1}}\otimes\dotsb\otimes e_{\nu_{a}}\otimes
\eta_{i'_{1}}\otimes\dotsb\otimes\eta_{i'_{j}}$. For such $u$ the
product of the values of $b$ is
$\prod_{p}\delta_{\nu_{p}\mu_{p}}\prod_{r}\delta_{i'_{r}i_{r}}$,
by $b(e_{i},e_{i'})=\delta_{ii'}$ and \eqref{eq:etapairing}, and
the Koszul sign of
$\beta_{d}$ is $(-1)^{j(j-1)/2}$, one factor $-1$ for each pair of
odd positions. Hence
$\beta_{d}(T_{d},w_{m})=(-1)^{j(j-1)/2}\,
T_{d}[e_{\mu_{1}}\otimes\dotsb\otimes\eta_{i_{j}}]$, and the two
signs cancel.
\end{proof}

The lemma identifies $\parinv$ as the Koszul sign of $\beta_{d}$
on $j$ odd factors, the sign of interleaving, or equally of
reversing, $j$ odd letters, and it exhibits $h^{\mathrm{RS}}$ as a
coordinate reading of $T_{d}$ in the letters $e_{i}$ and
$\eta_{i}$. The colouring-dependent signs of
Definition~\ref{defi:mpf}, the partner signs of the $\eta_{i}$ and
the sign of the sorting permutation, arise separately, when the
$\kappa$-paired wedge of Section~\ref{subsec:mpf} is expanded in
the sorted basis.

A purely odd instance exercises the mate, the transfer and the
dictionary together. Let $V=\mathbb{C}^{0|2}$, so that
$b(\xi_{1},\xi_{2})=1$, $\eta_{1}=-\xi_{2}$ and $\eta_{2}=\xi_{1}$,
and let $T_{2}=\xi_{1}\otimes\xi_{2}-\xi_{2}\otimes\xi_{1}$, an
even $S_{2}$-invariant tensor. Then
\[
h^{\mathrm{cat}}(\xi_{1}\wedge\xi_{2})
=\beta_{2}(T_{2},\xi_{1}\otimes\xi_{2})
=(-1)^{1}\cdot(-1)\,b(\xi_{2},\xi_{1})\,b(\xi_{1},\xi_{2})=-1,
\qquad
h^{\mathrm{RS}}(\xi_{1}\wedge\xi_{2})=1,
\]
the latter also by the dictionary, since
$T_{2}=\eta_{1}\otimes\eta_{2}-\eta_{2}\otimes\eta_{1}$. On the
loop graph $L$ the two networks agree, as
Lemma~\ref{lemm:paritytransfer} requires: here
$\copair=\copair_{1}=-T_{2}$ and $\twistedC=T_{2}$, so
\[
\mathrm{TN}_{h^{\mathrm{cat}},\copair}(L)
=h^{\mathrm{cat}}\bigl(\symq{2}(-T_{2})\bigr)
=h^{\mathrm{cat}}(-2\,\xi_{1}\wedge\xi_{2})=2
=h^{\mathrm{RS}}(2\,\xi_{1}\wedge\xi_{2})
=\mathrm{TN}_{h^{\mathrm{RS}},\twistedC}(L),
\]
whereas the free circle evaluates to $k-2\ell=-2$.

The dictionary reads $T_{d}$ in the letters $\eta_{i}$. Reading
it in the letters $\xi_{i}$ instead gives a second functional,
which represents the same graph parameter. Let
$h^{\xi}\in(\SymV V_{0}\otimes\Lambda V_{1})^{*}$ assign to each
sorted monomial $m$ the coefficient of $T_{d}$ at the word
$w_{m}$ in the basis of $V^{\otimes d}$ formed by the tensor
words in the letters $e_{i}$ and $\xi_{i}$; it vanishes in odd
exterior degree, since $T_{d}$ is even. Let $\Psi$ be the even
automorphism of $V$ fixing $V_{0}$ and sending $\xi_{i}$ to
$\eta_{i}$.

\begin{lemm}[Change of letters]\label{lemm:letters}
$h^{\mathrm{RS}}=h^{\xi}\circ\SymV_{s}(\Psi)$, and
$\mathrm{TN}_{h^{\mathrm{RS}},\twistedC}(G)
=\mathrm{TN}_{h^{\xi},\twistedC}(G)$ for every closed graph $G$
without free circles.
\end{lemm}

\begin{proof}
Let $\langle-,-\rangle$ be the bilinear form on $V^{\otimes d}$
for which the words in the letters $e_{i},\xi_{i}$ are
orthonormal; the words in the letters $e_{i},\eta_{i}$ are
orthonormal for it as well, each $\eta_{i}$ being a signed $\xi$,
so $\langle T_{d},-\rangle$ reads the coefficient of $T_{d}$ at
every word of either kind. Since $T_{d}$ is $S_{d}$-invariant,
its coefficient at a rearranged word is the Koszul sign times its
coefficient at the original word, and this is the sign that
$\symq{d}$ attaches when it sorts a word; hence
$h^{\xi}\circ\symq{d}=\langle T_{d},-\rangle$, and therefore
$h^{\xi}(\SymV_{s}(\Psi)\,m)
=\langle T_{d},\Psi^{\otimes d}w_{m}\rangle
=h^{\mathrm{RS}}(m)$ by Lemma~\ref{lemm:dictionary}, that is,
$h^{\mathrm{RS}}=h^{\xi}\circ\SymV_{s}(\Psi)$. Next, $\Psi$ is an
isometry of $(V,b)$, by the coordinate table of
Section~\ref{subsec:formcoords} and the values
$b(\eta_{i},\eta_{j})$ computed in the proof of
Lemma~\ref{lemm:fusion}, so it fixes $\copair$, which
$L_{\copair}=\mathrm{id}_{V}$ determines, and it fixes
$\twistedC=(\parityop\otimes\mathrm{id})\copair$ because it
commutes with $\parityop$; being even, it commutes with the
symmetries of $\svect$ too. Applying $\Psi^{\otimes2|E|}$ to
$\bigotimes_{e}\twistedC$ in the network
$\mathrm{TN}_{h^{\xi},\twistedC}(G)$ therefore changes nothing,
while moving it through $\mathrm{Regroup}_{G}$ to the vertices
replaces each $h^{\xi}\circ\symq{\deg v}$ by
$h^{\xi}\circ\SymV_{s}(\Psi)\circ\symq{\deg v}
=h^{\mathrm{RS}}\circ\symq{\deg v}$.
\end{proof}

\section{The circuit expansion and the Regts--Sevenster
formula}\label{sec:circuit}

Two tasks remain: to prove that the twisted tensor network of
Section~\ref{sec:extraction} is \emph{exactly} the Regts--Sevenster
sum of Definition~\ref{defi:mpf}, circuit signs included; and to assemble the
proof of the Main Theorem. The first
part is a statement about super linear algebra alone --- no
connection category enters --- and we state it as a theorem in its
own right, since it is the step that turns an abstract fibre
functor into the concrete model of \cite{RS21}.

\needspace{11\baselineskip}
\begin{theo}[Reconstruction]\label{theo:reconstruction}
Let $V=V_{0}\oplus V_{1}$ be a finite-dimensional super vector
space with a nondegenerate supersymmetric even bilinear form $b$,
coordinatised as in Section~\ref{subsec:formcoords}, let
$\twistedC$ be the twisted element, and let $h=(h_{d})$ be
functionals on the super symmetric powers of $V$ vanishing in odd
exterior degree. Then for every graph $G$ without free circles the
network equals the sum \eqref{eq:mpf} of
Definition~\ref{defi:mpf}:
\[
\mathrm{TN}_{h,\twistedC}(G)=p_{h}(G).
\]
\end{theo}

The proof occupies
Sections~\ref{subsec:reduction}--\ref{subsec:circuitproof}. The
strategy is to expand the twisted element into its even and odd
parts, observe that the even part is diagonal and reproduces the
colouring sum over $\psi$, and reduce the odd part to a sign
computation along each circuit; the computation is carried out for
a single circuit and shown to be insensitive to everything else.

\subsection{The reduction}\label{subsec:reduction}

In the edge order fixed for the network, expand each factor as
$\twistedC=\twistedC_{0}+\twistedC_{1}$ and distribute:
\[
\bigotimes_{e\in E}\twistedC
=\sum_{F\subseteq E}\;
\bigotimes_{e\in E}\twistedC_{\mathbf{1}_{F}(e)},
\]
where $\mathbf{1}_{F}$ is the characteristic function of $F$, so
the $F$-summand carries $\twistedC_{1}$ on the edges of $F$ and
$\twistedC_{0}$ elsewhere; both parts being even tensors, the
regroupings below introduce no sign on this account.
The even part $\twistedC_{0}=\copair_{0}=\sum_{i}e_{i}\otimes
e_{i}$ is diagonal, and its legs contribute, at each vertex, the
word of $e_{\psi(a)}$ over the half-edges
$a\in\halfedges{E\setminus F}{v}$, summed over colourings $\psi$
of the edges of $E\setminus F$; this is precisely the
$\psi$-sum of Definition~\ref{defi:mpf}. Since each $h_{d}$ vanishes in odd
exterior degree, a term survives only if every vertex receives an
even number of odd half-edges, i.e.\ only if $F$ is Eulerian.
For Eulerian $F$, choose an Eulerian orientation and a compatible
local pairing $\kappa$, as in
Section~\ref{subsec:mpf}; by \cite[Proposition~3]{RS21} the
right-hand side of the Reconstruction Theorem is independent of
these choices, and the argument below, which applies to an
arbitrary choice, re-derives this. By
Lemma~\ref{lemm:TNinv} we may also choose the
half-edge order at each vertex so that the even legs occur first
and the odd legs occur in $\kappa$-paired order; the even and odd
legs then combine through the quotient map without interaction:
for even vectors $e_{i_{1}},\dotsc,e_{i_{a}}$ and odd vectors
$x_{1},y_{1},\dotsc,x_{m},y_{m}$,
\begin{equation}\label{eq:mixedquotient}
\symq{d}\bigl(e_{i_{1}}\otimes\dotsb\otimes e_{i_{a}}
\otimes x_{1}\otimes y_{1}\otimes\dotsb\otimes x_{m}\otimes
y_{m}\bigr)
=(e_{i_{1}}\odot\dotsb\odot e_{i_{a}})\otimes
(x_{1}\wedge y_{1})\wedge\dotsb\wedge(x_{m}\wedge y_{m})
\end{equation}
under the isomorphism
$\SymV_{s}V\cong\SymV V_{0}\otimes\Lambda V_{1}$: moving even legs
past odd legs costs no sign, and the right-hand side is exactly
the vertex argument of
Definition~\ref{defi:mpf}. Write $\mathrm{Regroup}_{F,\kappa}$ for the
restriction of $\mathrm{Regroup}_{G}$ to the odd legs, targeting
the $\kappa$-paired vertex--half-edge order. The theorem then
reduces to the identity of Lemma~\ref{lemm:circuitsign}, applied
circuit by circuit.

\subsection{The circuit sign}\label{subsec:circuitproof}

Consider a single $\kappa$-circuit of length $n$: edges
$a_{1},\dotsc,a_{n}$ traversed in order through
vertices $v_{1},\dotsc,v_{n}$, the head half-edge $x_{j}$ of
$a_{j}$ landing in the in-slot of $v_{j}$ and the tail half-edge
$y_{j}$ in the out-slot of $v_{j-1}$ (indices modulo $n$). The
$v_{j}$ denote successive visits and need not be distinct: a
vertex visited several times, by this circuit or by several,
receives one degree-two factor per visit, and these even factors
are collected afterwards by exterior multiplication, which
introduces no sign.

\begin{lemm}\label{lemm:cyclicparity}
The permutation from the edge-wise word
$(x_{1},y_{1},x_{2},y_{2},\dotsc,x_{n},y_{n})$ to the vertex-wise
word $(x_{1},y_{2}\,|\,x_{2},y_{3}\,|\,\dotsc\,|\,x_{n},y_{1})$
fixes the positions $1,3,\dotsc,2n-1$ and acts on the positions
$2,4,\dotsc,2n$ as a
single $n$-cycle. All $2n$ legs being odd, its Koszul sign is
$(-1)^{n-1}$.
\end{lemm}

\begin{proof}
The $x_{j}$ occupy the positions $1,3,\dotsc,2n-1$ in both
words, in the same
order; the $y$'s move by one step of a cyclic rotation, which is
an $n$-cycle on the remaining positions, of sign $(-1)^{n-1}$; by
Lemma~\ref{lemm:koszul} the Koszul sign of a rearrangement of an
all-odd word is the sign of the permutation. We illustrate the
two words in Figure~\ref{fig:circuit}.
\end{proof}

\begin{figure}
\centering
\begin{tikzpicture}[scale=0.9,
  vert/.style={circle,fill,inner sep=1.6pt},
  l/.style={inner sep=1.5pt,font=\small}]
  \node[vert] (v1) at (90:1) {};
  \node[vert] (v2) at (210:1) {};
  \node[vert] (v3) at (330:1) {};
  \draw[->] (v1) -- node[left,font=\scriptsize] {$a_{2}$} (v2);
  \draw[->] (v2) -- node[below,font=\scriptsize] {$a_{3}$} (v3);
  \draw[->] (v3) -- node[right,font=\scriptsize] {$a_{1}$} (v1);
  \node[font=\scriptsize,above] at (v1.north) {$v_{1}$};
  \node[font=\scriptsize,below left] at (v2.south) {$v_{2}$};
  \node[font=\scriptsize,below right] at (v3.south) {$v_{3}$};
  \begin{scope}[xshift=3.1cm,yshift=0.75cm]
    \node[l,anchor=east] at (1.0,0) {edge-wise:};
    \node[l] (ty1) at (2.0,0) {$y_{1}$};
    \node[l] (tx1) at (1.5,0) {$x_{1}$};
    \node[l] (tx2) at (2.5,0) {$x_{2}$};
    \node[l] (ty2) at (3.0,0) {$y_{2}$};
    \node[l] (tx3) at (3.5,0) {$x_{3}$};
    \node[l] (ty3) at (4.0,0) {$y_{3}$};
    \node[l,anchor=east] at (1.0,-1.6) {vertex-wise:};
    \node[l] (bx1) at (1.5,-1.6) {$x_{1}$};
    \node[l] (by2) at (2.0,-1.6) {$y_{2}$};
    \node[l] at (2.25,-1.6) {$|$};
    \node[l] (bx2) at (2.5,-1.6) {$x_{2}$};
    \node[l] (by3) at (3.0,-1.6) {$y_{3}$};
    \node[l] at (3.25,-1.6) {$|$};
    \node[l] (bx3) at (3.5,-1.6) {$x_{3}$};
    \node[l] (by1) at (4.0,-1.6) {$y_{1}$};
    \draw[->,gray] (ty1.south) to[out=-60,in=120] (by1.north);
    \draw[->,gray] (ty2.south) to[out=-120,in=60] (by2.north);
    \draw[->,gray] (ty3.south) to[out=-120,in=60] (by3.north);
  \end{scope}
\end{tikzpicture}
\caption{Edge-wise and vertex-wise words on a circuit of length
three.}
\label{fig:circuit}
\end{figure}
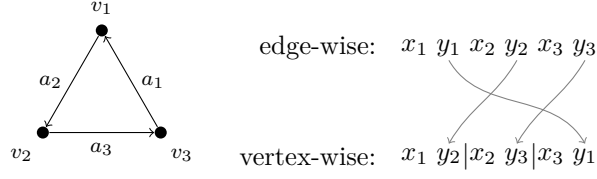
\FloatBarrier

Three small auxiliaries dispose of the interactions between
circuits. First, the quotient map sends an interleaved product of
odd pairs to the wedge --- the all-odd case of
\eqref{eq:mixedquotient}. Second, each pair
$x\wedge y$ of odd vectors is an even block, so reordering local
pairs at a vertex, and moving complete circuit blocks past one
another --- including circuits sharing a vertex --- introduces no
sign. Third, each factor $\twistedC_{1}$ is a sum of terms with
two odd legs, hence an even block, so reordering the factors of
$\twistedC_{1}^{\otimes|F|}$ into circuit-consecutive order costs
no sign; the edge order in Lemma~\ref{lemm:cyclicparity} may
therefore be assumed circuit-consecutive.\footnote{This
normalisation is the only point at which the circuit
decomposition enters the proof, and it can be avoided: writing
$W$ for the walk permutation of the half-edges of $F$ determined
by $\kappa$ and the orientation, the orbits of
$W|_{\mathrm{out}}$, its restriction to outgoing half-edges, are
the circuits, so
$\operatorname{sgn}(W|_{\mathrm{out}})=(-1)^{|F|-c(\kappa)}$,
i.e.\ $(-1)^{c(\kappa)}=(-1)^{|F|}
\operatorname{sgn}(W|_{\mathrm{out}})$; the total regrouping
parity is the same permutation sign, and the two combine as in
Lemma~\ref{lemm:circuitsign} with no decomposition into circuits.
We keep the circuitwise proof for its transparency.}

\begin{lemm}\label{lemm:circuitsign}
With notation as above,
\[
\Bigl(\bigotimes_{v}\symq{\deg_{F}v}\Bigr)\circ
\mathrm{Regroup}_{F,\kappa}
\bigl(\twistedC_{1}^{\otimes|F|}\bigr)
=(-1)^{c(\kappa)}
\sum_{\varphi\colon F\to[2\ell]}\;
\bigotimes_{v}\bigwedge_{(a_{1},a_{2})\in\kappa_{v}}
\xi_{\varphi(a_{1})}\wedge\eta_{\varphi(a_{2})} .
\]
\end{lemm}

\begin{proof}
By the third auxiliary the edge factors are in
circuit-consecutive order, and by the second the identity
factors over the circuits of $\kappa$, the degree-two factors of
the visits at a common vertex being multiplied together at the
end; fix one circuit, of length $n$. Writing each of its edge factors as
$\twistedC_{1}=-\sum_{i}\xi_{i}\otimes\eta_{i}$
(Section~\ref{subsec:copairing}), with the head leg carrying
$\xi$ and the tail leg carrying $\eta$ --- an assignment licensed
by the supersymmetry of $\twistedC$
(Section~\ref{subsec:copairing}) --- the $n$ edge factors
contribute the sign $(-1)^{n}$ in total. The regrouping to the
vertex-wise word contributes $(-1)^{n-1}$
(Lemma~\ref{lemm:cyclicparity}). The resulting vertex-wise word
is $\bigotimes_{j}\xi_{\varphi(a_{j})}\otimes
\eta_{\varphi(a_{j+1})}$, constant colourings and mixed
colourings alike being carried along; the quotient maps convert
it to the displayed wedges by the first auxiliary, with no
further sign. The total sign per circuit is
$(-1)^{n}(-1)^{n-1}=-1$, independent of $n$, giving
$(-1)^{c(\kappa)}$ overall. Loops are the case $n=1$ and pairs of
parallel edges the case $n=2$, by the half-edge conventions of
Section~\ref{subsec:graphs}.
\end{proof}

\begin{proof}[Proof of the Reconstruction Theorem]
Combine Lemma~\ref{lemm:circuitsign} with the reduction of
Section~\ref{subsec:reduction}: for each Eulerian $F$, the odd
sector produces the sign $(-1)^{c(\kappa)}$ together with the
$\varphi$-sum of wedges, the even sector produces the
$\psi$-sum, and applying the functionals $h_{\deg v}$ at the
vertices yields exactly the summand of Definition~\ref{defi:mpf} for $F$.
\end{proof}

For the loop graph $L$ of Section~\ref{subsec:mpf}, the case
$n=1$ of Lemma~\ref{lemm:circuitsign} reproduces the hand
computation of Section~\ref{subsec:mpf} term by term. A check of
a different kind concerns two circuits at a common vertex, which
a single loop leaves untested. Take one vertex with two loops $a$
and $b$, both
in $F$, with $\ell\ge2$ and colours $\varphi(a)=i$ and
$\varphi(b)=j$ from different symplectic pairs, so that the wedge
of exterior degree four is nonzero; orient each loop. The local
pairing that joins each loop to itself has two circuits and vertex
argument $\xi_{i}\wedge\eta_{i}\wedge\xi_{j}\wedge\eta_{j}$; the
one that pairs the incoming half-edge of each loop with the
outgoing half-edge of the other has one circuit and vertex
argument $\xi_{i}\wedge\eta_{j}\wedge\xi_{j}\wedge\eta_{i}$. Moving
$\eta_{i}$ leftward past $\xi_{j}$ and $\eta_{j}$ costs two
transpositions and no sign, and the final exchange of $\eta_{j}$
and $\xi_{j}$ costs one, so
$\xi_{i}\wedge\eta_{j}\wedge\xi_{j}\wedge\eta_{i}
=-\xi_{i}\wedge\eta_{i}\wedge\xi_{j}\wedge\eta_{j}$. The passage
from two circuits to one flips the circuit sign, the change of
wedge flips it back, and the two summands of Definition~\ref{defi:mpf} agree,
as \cite[Proposition~3]{RS21} guarantees; on the network side the
two pairings are presentations of one choice-free value.

\subsection{Multiplicativity, circles, and the proof of the Main
Theorem}\label{subsec:assembly}

The mixed partition function of Definition~\ref{defi:mpf} is
multiplicative over disjoint union: the Eulerian
subset, both colourings, and the local circuit data of a disjoint
union all split componentwise, and $c(\kappa)$ is additive; the
circle convention extends this to graphs with free circles.

\begin{proof}[Proof of the Main Theorem]
If $f=p_{h}$ with $k+2\ell\le R$, then
$\rk\connmat{f}{t}\le\max(1,k+2\ell)^{t}\le R^{t}$ for all
$t\ge0$ by \cite[Theorem~6]{RS21}, as recalled in
Section~\ref{subsec:mpf}. For the converse, let $f$
satisfy (H1) and (H2). Write an arbitrary graph
as $G=G^{\circ}\sqcup\freecircle^{\,r}$ with $G^{\circ}$ free of
circles. Then
\begin{align*}
f(G)&=f(G^{\circ})\,f(\freecircle)^{r}
&&\text{(Lemma~\ref{lemm:mult})}\\
&=(k-2\ell)^{r}\,
\mathrm{TN}_{h^{\mathrm{cat}},\copair}(G^{\circ})
&&\text{(Section~\ref{subsec:copairing},
Lemma~\ref{lemm:indexraising})}\\
&=(k-2\ell)^{r}\,
\mathrm{TN}_{h^{\mathrm{RS}},\twistedC}(G^{\circ})
&&\text{(Lemma~\ref{lemm:paritytransfer})}\\
&=(k-2\ell)^{r}\,p_{h^{\mathrm{RS}}}(G^{\circ})
&&\text{(the Reconstruction Theorem)}\\
&=p_{h^{\mathrm{RS}}}(G)
&&\text{(circle convention),}
\end{align*}
so $f=p_{h^{\mathrm{RS}}}$ is a mixed partition function on
$\mathbb{C}^{k|2\ell}$, and $k+2\ell\le R$ by
Corollary~\ref{coro:dimbound}.
\end{proof}

\section{Minimal models and prescribed dimensions}\label{sec:consequences}

We close with the intrinsic content of the dimension bound. Except
in
Theorem~\ref{theo:prescribed} and Corollary~\ref{coro:schrijver},
whose hypotheses are stated in full, $f$ satisfies (H1) and (H2)
throughout, $V=\fib(\genX)$ is the super vector space of
Section~\ref{subsec:fibrefunctor}, with $\dim V_{0}=k$ and
$\dim V_{1}=2\ell$, and $h^{\mathrm{RS}}$ is the functional of
Section~\ref{subsec:paritytransfer}, so that
$f=p_{h^{\mathrm{RS}}}$ by the Main Theorem. We call
$(k,\ell,h^{\mathrm{RS}})$ the \emph{reconstructed model} of $f$,
and any triple $(k',\ell',h')$ with
$h'\in(\SymV\mathbb{C}^{k'}\otimes\Lambda\mathbb{C}^{2\ell'})^{*}$
and $f=p_{h'}$ a \emph{model} of $f$; the pair $(k',2\ell')$ is
the model's \emph{dimension pair} and $k'+2\ell'$ its \emph{total
number of colours}.

\subsection{The dimensions of the models}\label{subsec:growthbase}

Recall the growth base
\[
\growthbase=
\sup_{t\ge1}\bigl(\rk\connmat{f}{t}\bigr)^{1/t},
\]
the edge-rank connectivity of \cite{RS21} with the value zero
allowed; a real number $R\ge1$ satisfies (H2) iff
$R\ge\growthbase$, since the case $t=0$ of (H2) does not involve
$R$.

\begin{lemm}[Growth]\label{lemm:growthrate}
Let $d\coloneqq k+2\ell$ be the total number of colours of the
reconstructed model. Then $\rk\connmat{f}{2n}\le d^{2n}$ for
every $n\ge1$, and, when $d\ge1$,
\[
d^{2n}\le\binom{n+d^{2}-1}{d^{2}-1}\,\rk\connmat{f}{2n};
\]
consequently $(\rk\connmat{f}{2n})^{1/2n}\to d$ as $n\to\infty$,
and $\growthbase=d$.
\end{lemm}

\begin{proof}
The first inequality is \cite[Theorem~6]{RS21} applied to the
reconstructed model, $f=p_{h^{\mathrm{RS}}}$, and the second is
Lemma~\ref{lemm:diminequality}. If $d=0$, the same theorem gives
$\rk\connmat{f}{t}=0$ for every $t\ge1$, so the limit and
$\growthbase$ both vanish. If $d\ge1$, taking $2n$-th roots gives
the limit, since
$\bigl(\tbinom{n+d^{2}-1}{d^{2}-1}\bigr)^{1/2n}\to1$; hence
$\growthbase\ge d$, while $\growthbase\le d$ because
$\rk\connmat{f}{t}\le d^{t}$ for every $t\ge1$ by the same
theorem.
\end{proof}

\begin{lemm}[Padding]\label{lemm:padding}
Let $h_{1}\in(\SymV\mathbb{C}^{k_{1}}\otimes
\Lambda\mathbb{C}^{2\ell_{1}})^{*}$, and let $k_{2}\ge k_{1}$ and
$\ell_{2}\ge\ell_{1}$ satisfy $k_{2}-2\ell_{2}=k_{1}-2\ell_{1}$.
Then there is $h_{2}\in(\SymV\mathbb{C}^{k_{2}}\otimes
\Lambda\mathbb{C}^{2\ell_{2}})^{*}$ with $p_{h_{2}}=p_{h_{1}}$.
\end{lemm}

\begin{proof}
Embed $\mathbb{C}^{k_{1}|2\ell_{1}}$ into
$\mathbb{C}^{k_{2}|2\ell_{2}}$ by $e_{i}\mapsto e_{i}$ for
$i\le k_{1}$ and, for $m\le\ell_{1}$, by $\xi_{m}\mapsto\xi_{m}$
and $\xi_{m+\ell_{1}}\mapsto\xi_{m+\ell_{2}}$; write $\iota$ both
for this embedding and for the induced injections of the colour
sets. The embedding respects the signed partners of
Section~\ref{subsec:mpf}: $\iota(\eta_{i})=\eta_{\iota(i)}$ for
every odd colour $i$, the right-hand side being formed in
$\mathbb{C}^{k_{2}|2\ell_{2}}$, as one checks separately for
$i\le\ell_{1}$ and for $i>\ell_{1}$. Let $h_{2}$ agree with
$h_{1}$ on the image of $\iota$ and vanish on every basis monomial
of $\SymV\mathbb{C}^{k_{2}}\otimes\Lambda\mathbb{C}^{2\ell_{2}}$
that involves a basis vector outside the image. In the sum of
Definition~\ref{defi:mpf} for $p_{h_{2}}(G)$, a colouring that uses a colour
outside the image contributes zero: an even such colour places a
basis vector outside the image in the symmetric factor at both
endpoints of its edge, and an odd one places $\xi_{j}$ or
$\eta_{j}$ in the exterior factor at an endpoint, and both lie
outside the image, which is closed under partners. The remaining
colourings are the images under $\iota$ of the colourings of
Definition~\ref{defi:mpf} for $h_{1}$, and the partner identity shows that
their vertex arguments correspond term by term; so
$p_{h_{2}}(G)=p_{h_{1}}(G)$ for every circle-free $G$. The values
on free circles agree by $k_{2}-2\ell_{2}=k_{1}-2\ell_{1}$.
\end{proof}

The following theorem contains Theorem~\ref{theo:introdims}.

\begin{theo}[Model dimensions]\label{theo:modeldims}
Let $(k,\ell,h^{\mathrm{RS}})$ be the reconstructed model of $f$.
\begin{enumerate}
\item[(a)] $\growthbase=k+2\ell
  =\lim_{n\to\infty}(\rk\connmat{f}{2n})^{1/2n}$; in particular
  $\growthbase$ is a nonnegative integer.
\item[(b)] $k=\bigl(\growthbase+f(\freecircle)\bigr)/2$ and
  $2\ell=\bigl(\growthbase-f(\freecircle)\bigr)/2$; in particular
  $f(\freecircle)$ is an integer with
  $|f(\freecircle)|\le\growthbase$,
  $\growthbase+f(\freecircle)\in2\mathbb{N}$ and
  $\growthbase-f(\freecircle)\in4\mathbb{N}$.
\item[(c)] The dimension pairs of the models of $f$ are exactly
  the pairs $(k+2r,\,2\ell+2r)$ with $r\in\mathbb{N}$. In
  particular every model of $f$ has at least $k$ even and at least
  $2\ell$ odd colours, the total numbers of colours of the models
  of $f$ are exactly the numbers $\growthbase+4r$ with
  $r\in\mathbb{N}$, and $\growthbase$ is the least of them.
\end{enumerate}
\end{theo}

\begin{proof}
Part (a) is Lemma~\ref{lemm:growthrate}. For (b),
$f(\freecircle)=k-2\ell$ by Section~\ref{subsec:copairing}, and
solving the two equations $k+2\ell=\growthbase$ and
$k-2\ell=f(\freecircle)$ gives the formulas; the remaining
assertions read them off. For (c), let $(k',\ell',h')$ be a model
of $f$. By \cite[Theorem~6]{RS21} applied to $h'$,
$\rk\connmat{f}{t}\le(k'+2\ell')^{t}$ for all $t\ge1$, so
$k+2\ell=\growthbase\le k'+2\ell'$; and
$k'-2\ell'=p_{h'}(\freecircle)=f(\freecircle)=k-2\ell$ by the
circle convention of Section~\ref{subsec:mpf}. Adding the two
relations gives $k'\ge k$, and subtracting them gives
$k'-k=2(\ell'-\ell)$; so $(k',2\ell')=(k+2r,2\ell+2r)$ with
$r\coloneqq\ell'-\ell\ge0$. Conversely, for every $r\in\mathbb{N}$
the pair $(k+2r,2\ell+2r)$ satisfies the balance condition of
Lemma~\ref{lemm:padding}, which pads the reconstructed model to a
model with that dimension pair.
\end{proof}

Theorem~\ref{theo:modeldims} determines the dimensions of the
models, not the models themselves: minimal models are not unique,
even up to changes of coordinates preserving the form. On
$\mathbb{C}^{2|0}$ with its
standard form, let $h_{0}$ and $h_{u}$ both take the value $1$ in
degree zero and vanish in every degree at least two, and in degree
one let $h_{0}$ vanish while $h_{u}(e_{1})=1$ and
$h_{u}(e_{2})=\sqrt{-1}$. A vertex of degree at least two makes
either vertex factor vanish, so both partition functions vanish on
every graph with such a vertex; on a component consisting of one
edge and its two endpoints they take the values $0^{2}+0^{2}$ and
$1^{2}+(\sqrt{-1})^{2}$, both zero; and both are $1$ on edgeless
graphs and $2$ on a free circle. Hence $p_{h_{0}}=p_{h_{u}}$, and
both models are minimal, the free-circle value $2$ forcing at
least two colours; yet no change of coordinates carries the zero
functional $h_{0}$ in degree one to the nonzero functional
$h_{u}$. The degree-one part of $h_{u}$ is an isotropic vector, so
the orbit of $h_{u}$ under the orthogonal group accumulates at
$h_{0}$, and the partition function, a continuous invariant of the
orbit, cannot separate the two.

The limit in (a) is taken along even arities, because the
symmetric-group actions that produce the lower bound live on
$\End(\genX^{\otimes n})$, whose dimension is the rank at arity
$2n$; and it cannot be taken along all arities. Consider the
model on $\mathbb{C}^{1|0}$ whose functional is $1$ on
$e_{1}^{\odot d}$ for even $d$ and $0$ for odd $d$: its partition
function is $1$ on the circle-free graphs all of whose degrees are
even and $0$ on the others, and $1$ on a free circle. The internal
vertices of a $t$-fragment have degrees summing to $2|E|-t$, where
$E$ is the set of edges other than free circles, and gluing does
not change these degrees, so for odd $t$ every closure
of a $t$-fragment has a vertex of odd degree, and
$\rk\connmat{f}{t}=0$; while $\rk\connmat{f}{2n}=1$ for every
$n$, the closure of $n$ strands against $n$ strands being $n$ free
circles. Here $\growthbase=1$, and $(\rk\connmat{f}{t})^{1/t}$
vanishes for every odd $t$.

\subsection{An example: the characteristic polynomial}
\label{subsec:charpoly}

We return to the example of Section~\ref{subsec:mpf}: for
$\theta\in\mathbb{C}$, the parameter $p_{h(\theta)}$ is
$G\mapsto\det(\theta I-A_{G})$ on circle-free graphs and vanishes
on graphs with a free circle, a model on $\mathbb{C}^{2|2}$ with
free-circle value $0$. The statement concerns this parameter on
the multigraphs of Section~\ref{subsec:graphs}, free circles
included; its proof uses loops.

\begin{coro}[The characteristic polynomial]\label{coro:charpoly}
For every $\theta\in\mathbb{C}$, the dimension pairs of the models
of $p_{h(\theta)}$ are exactly the pairs $(2+2r,2+2r)$ with
$r\in\mathbb{N}$: the model of Regts and Sevenster on
$\mathbb{C}^{2|2}$ is minimal, the growth base is $4$, and
$(\rk\connmat{p_{h(\theta)}}{2n})^{1/2n}\to4$.
\end{coro}

\begin{proof}
Let $k$ and $2\ell$ be the even and odd dimensions of the
reconstructed model of $p_{h(\theta)}$. By
Theorem~\ref{theo:modeldims}(b), $k=2\ell$
and $\growthbase=4\ell$, and by (c) the dimension pairs of the
models are the $(2\ell+2r,2\ell+2r)$ with $r\in\mathbb{N}$. The
model with no colours vanishes on every graph with an edge,
whereas for the graph $L_{m}$ with one vertex and $m$ loops, whose
adjacency matrix is $(2m)$, $p_{h(\theta)}(L_{m})=\theta-2m$, and
$\theta-2$ and $\theta-4$ do not both vanish; so $\ell\ge1$. The
model of Section~\ref{subsec:mpf} has dimension pair $(2,2)$,
which is of the form $(2\ell+2r,2\ell+2r)$ only for $\ell=1$ and
$r=0$. The limit is Theorem~\ref{theo:modeldims}(a).
\end{proof}

The bound $k+2\ell\le R$ of the Main Theorem is thus attained,
with $R=4$, by the characteristic polynomial at every value of
$\theta$; and the growth base $4$ beside the circle value $0$
shows rank growth detecting the least total dimension of a model,
and the free circle the superdimension of every model.

\subsection{Prescribed dimensions}\label{subsec:prescribed}

Theorem~\ref{theo:modeldims}(c) describes the models of a
parameter already known to satisfy (H1) and (H2). We now
characterise the parameters admitting a model on a prescribed
$\mathbb{C}^{k'|2\ell'}$, in two forms: one in terms of connection
ranks, valid for every normalised parameter, and one in terms of
the connection category.

\begin{theo}[Prescribed dimensions]\label{theo:prescribed}
Let $f$ be a graph parameter with $f(\emptygraph)=1$, and let
$k',\ell'\in\mathbb{N}$. Then $f=p_{h'}$ for some
$h'\in(\SymV\mathbb{C}^{k'}\otimes\Lambda\mathbb{C}^{2\ell'})^{*}$
iff
\[
f(\freecircle)=k'-2\ell'
\qquad\text{and}\qquad
\rk\connmat{f}{t}\le(k'+2\ell')^{t}\ \text{ for all }t\ge0,
\]
with $0^{0}=1$.
\end{theo}

The case $t=0$ of the rank condition reads
$\rk\connmat{f}{0}\le1$, which with $f(\emptygraph)=1$ is the
multiplicativity of $f$ (Lemma~\ref{lemm:mult}); the theorem
includes $k'=\ell'=0$, where the conditions say that $f$ is
multiplicative and vanishes on every graph with an edge or a
circle.

\begin{proof}
If $f=p_{h'}$, then $f(\freecircle)=k'-2\ell'$ by the circle
convention, and the rank condition is \cite[Theorem~6]{RS21}, as
recalled in Section~\ref{subsec:mpf}. Conversely, the rank
condition is (H2) with $R=\max(1,k'+2\ell')$, so $f$ has a
reconstructed model $(k,\ell,h^{\mathrm{RS}})$ and
Theorem~\ref{theo:modeldims} applies: $\growthbase\le k'+2\ell'$
by the definition of $\growthbase$, so $k+2\ell\le k'+2\ell'$ by
(a), while $k-2\ell=f(\freecircle)=k'-2\ell'$ by (b). As in the
proof of (c), $(k',2\ell')=(k+2r,2\ell+2r)$ for some
$r\in\mathbb{N}$, and (c) provides a model with that dimension
pair.
\end{proof}

For the second form recall from Section~\ref{subsec:fibrefunctor}
that a partition $\lambda\vdash n$ \emph{survives} for $\genX$ if
$\Phi_{n}(e_{\lambda})\ne0$ in $\End(\genX^{\otimes n})$ and
\emph{dies} otherwise, and from Appendix~\ref{app:zeta} the fat
hook $\hook{a}{b}=\{\lambda:\lambda_{a+1}\le b\}$, outside which
lie exactly the partitions containing the rectangle
$\bigl((b+1)^{a+1}\bigr)$.

\begin{lemm}[Survival hook]\label{lemm:survivalhook}
A partition survives for $\genX$ iff it lies in $\hook{k}{2\ell}$.
\end{lemm}

\begin{proof}
By the hook criterion for super vector spaces
\cite[Cor.~1.9]{Del02} (in substance the theorem of Berele and
Regev \cite{BR87}), $\schur{\lambda}{V}\ne0$ exactly when
$\lambda\in\hook{k}{2\ell}$, and $\schur{\lambda}{V}\ne0$ iff
$\Phi^{V}_{n}(e_{\lambda})=\fib(\Phi_{n}(e_{\lambda}))$ is nonzero
(Section~\ref{subsec:fibrefunctor}). This is where the
faithfulness of $\fib$ is used: $\fib(\Phi_{n}(e_{\lambda}))\ne0$
iff $\Phi_{n}(e_{\lambda})\ne0$, one direction by linearity and
the other by faithfulness.
\end{proof}

\begin{coro}[Prescribed dimensions, categorical
form]\label{coro:rectangle}
Assume (H1) and (H2), and let $k',\ell'\in\mathbb{N}$. Then $f$
has a model on $\mathbb{C}^{k'|2\ell'}$ iff
$f(\freecircle)=k'-2\ell'$ and the rectangle
$\bigl((2\ell'+1)^{k'+1}\bigr)$ dies for $\genX$.
\end{coro}

\begin{proof}
A rectangle $\bigl((b+1)^{a+1}\bigr)$ lies outside
$\hook{k}{2\ell}$ iff $k\le a$ and $2\ell\le b$, so by
Lemma~\ref{lemm:survivalhook} the rectangle
$\bigl((2\ell'+1)^{k'+1}\bigr)$ dies iff $k\le k'$ and
$\ell\le\ell'$. Together with $f(\freecircle)=k'-2\ell'$, i.e.\
$k'-2\ell'=k-2\ell$, this says exactly that
$(k',2\ell')=(k+2r,2\ell+2r)$ for some $r\in\mathbb{N}$, which by
Theorem~\ref{theo:modeldims}(c) is the condition for a model.
\end{proof}

Within the class of parameters satisfying (H1) and (H2), the
corollary replaces the prescribed rank bounds
$\rk\connmat{f}{t}\le(k'+2\ell')^{t}$ at every arity by the
vanishing of one Schur idempotent at one arity, together with the
free-circle condition. The death of the rectangle means the
vanishing of every closure of
$\Phi_{n}(e_{((2\ell'+1)^{k'+1})})$ against the fragments of arity
$2n$, where $n=(k'+1)(2\ell'+1)$; since $\End(\genX^{\otimes n})$
is finite-dimensional, Lemma~\ref{lemm:negligible} reduces this
to a finite spanning family of fragments, which the argument
leaves implicit.

\subsection{Ordinary and skew models, and Schrijver's theorem}
\label{subsec:ordinaryskew}

A model is \emph{ordinary} if it has no odd colours and
\emph{skew} if it has no even colours; the skew models are the
symplectic models of \cite{RS17}.

\begin{coro}[Ordinary and skew models]\label{coro:ordinaryskew}
The following are equivalent:
\begin{enumerate}
\item[(i)] $f$ has an ordinary model;
\item[(ii)] $\ell=0$;
\item[(iii)] $\growthbase=f(\freecircle)$;
\item[(iv)] $\rk\connmat{f}{t}\le f(\freecircle)^{t}$ for all
  $t\ge1$;
\item[(v)] some super exterior power of $\genX$ vanishes, i.e.\
  $\Phi_{n}(e_{(1^{n})})=0$ for some $n\ge1$.
\end{enumerate}
Likewise, $f$ has a skew model iff $k=0$, iff
$\growthbase=-f(\freecircle)$, iff
$\rk\connmat{f}{t}\le(-f(\freecircle))^{t}$ for all $t\ge1$, iff
some super symmetric power of $\genX$ vanishes, i.e.\
$\Phi_{n}(e_{(n)})=0$ for some $n\ge1$.
\end{coro}

\begin{proof}
By Theorem~\ref{theo:modeldims}(c) the models of $f$ have
$\ell'=\ell+r$ with $r\ge0$, so one of them has $\ell'=0$ iff
$\ell=0$: (i) iff (ii). Since $\growthbase=k+2\ell$ and
$f(\freecircle)=k-2\ell$, (ii) iff (iii). For (iii) iff (iv):
$\growthbase\ge|f(\freecircle)|$ by
Theorem~\ref{theo:modeldims}(b), so if $f(\freecircle)<0$ both
conditions fail, (iv) already at $t=1$; and if
$f(\freecircle)\ge0$ then (iv) says $\growthbase\le f(\freecircle)$
by the definition of $\growthbase$, i.e.\
$\growthbase=f(\freecircle)$. Finally (ii) iff (v) by
Lemma~\ref{lemm:survivalhook}: $(1^{n})\notin\hook{k}{2\ell}$ iff
$n>k$ and $2\ell<1$, so some $(1^{n})$ dies iff $\ell=0$. The
skew case is the same with the roles of rows and columns
exchanged: $(n)\notin\hook{k}{2\ell}$ iff $k<1$ and $n>2\ell$.
\end{proof}

\begin{coro}[Schrijver's theorem]\label{coro:schrijver}
A graph parameter $f$ is the partition function of an ordinary
edge-colouring model iff $f(\emptygraph)=1$, $f(\freecircle)$ is
real, and $\rk\connmat{f}{t}\le f(\freecircle)^{t}$ for all
$t\ge0$.
\end{coro}

This is the theorem of \cite{Sch15}, in the form stated there;
the Main Theorem contains it as the sector $\ell=0$. The same
argument characterises the skew models by $f(\emptygraph)=1$,
$f(\freecircle)$ real and
$\rk\connmat{f}{t}\le(-f(\freecircle))^{t}$ for all $t\ge0$,
recovering, with rank bounds at all arities in place of the even
ones, the characterisation of Regts and Sevenster
\cite[Theorem~1]{RS17}; see also \cite[Theorem~4.3]{Sev18}.

\begin{proof}
An ordinary model with $k'$ colours has $f(\emptygraph)=1$,
$f(\freecircle)=k'$ and $\rk\connmat{f}{t}\le k'^{t}$ for all
$t\ge0$ \cite[Theorem~6]{RS21}. Conversely, let $f$ satisfy the
three conditions. At $t=1$ the rank condition reads
$\rk\connmat{f}{1}\le f(\freecircle)$, so $f(\freecircle)\ge0$,
and (H2) holds with $R=\max(1,f(\freecircle))$. Hence $f$ has a
reconstructed model, the rank condition at $t\ge1$ is condition
(iv) of Corollary~\ref{coro:ordinaryskew}, and $f$ has an
ordinary model.
\end{proof}

\subsection{Sevenster's conjecture}\label{subsec:sevenster}

Conjecture~5.7 of \cite{Sev18} proposes a characterisation of the
mixed partition functions of a fixed $\mathbb{C}^{k'|2\ell'}$ in
the spirit of the characterisation of the ordinary edge-colouring
models with $k'$ colours due to Draisma, Gijswijt, Lov\'asz, Regts
and Schrijver \cite[Theorem~1]{DGLRS12}, which carries no rank
hypothesis and which, with free circles admitted, reads as follows
\cite[Theorem~2.4]{Sev18}: $f$ is such a partition function iff
$f(\emptygraph)=1$, $f(\freecircle)=k'$, $f$ is multiplicative,
and $f$ vanishes on every graph antisymmetrised over the endpoints
of $k'+1$ of its edges. Sevenster's conjecture reads: $f$ is a
mixed partition function on $\mathbb{C}^{k'|2\ell'}$ iff
$f(\emptygraph)=1$, $f(\freecircle)=k'-2\ell'$, $f$ is
multiplicative, and $f$ vanishes on a subspace $J_{k',2\ell'}$ of
the graph algebra spanned by Young-symmetrised graphs
\cite[(2.12)--(2.13)]{Sev18}, one family for each \emph{even}
partition --- every part even --- whose diagram contains the cell
$(k'+1,2\ell'+1)$, that is, for each even partition outside the
fat hook $\hook{k'}{2\ell'}$. The forward implication is his
Theorem~4.6.

Corollary~\ref{coro:rectangle} is a statement of the same shape
under the rank hypothesis: given (H1) and (H2), $f$ is a mixed
partition function on $\mathbb{C}^{k'|2\ell'}$ iff
$f(\freecircle)=k'-2\ell'$ and the rectangle
$\bigl((2\ell'+1)^{k'+1}\bigr)$, the least partition outside the
hook, dies for $\genX$. The two conditions concern the same
forbidden hook but different partitions and different elements.
Sevenster's family consists of the even partitions outside the
hook, the least of which is the rectangle
$\bigl((2\ell'+2)^{k'+1}\bigr)$, and its elements are Young
symmetrisers; ours is the single rectangle
$\bigl((2\ell'+1)^{k'+1}\bigr)$, which is not even, and its
central idempotent. The choice of element is immaterial: a nonzero
element of the simple block $B_{\lambda}$ generates it, so the
image of a Young symmetriser of $\lambda$ under $\Phi_{n}$ vanishes
iff $\Phi_{n}(e_{\lambda})$ does. The substantive difference is
that Sevenster's condition asks that $f$ vanish on particular
closed graphs built from the Young symmetriser, whereas the death
of a partition in $\conncat$ is vanishing against every closure.
Under the rank hypothesis our condition implies his: given the
death of the rectangle and the circle value, $f$ has a model on
$\mathbb{C}^{k'|2\ell'}$ by Corollary~\ref{coro:rectangle}, and
then $f(J_{k',2\ell'})=0$ by his Theorem~4.6.

The converse implication, from Sevenster's vanishing relations to
the finiteness and growth of the morphism spaces of $\conncat$ or
to the death of the rectangle there, and with it the removal of
the rank hypothesis, lies beyond the Main Theorem; for $\ell'=0$
the conjecture is the theorem of \cite{DGLRS12}, whose
invariant-theoretic proof carries no rank hypothesis.

\subsection{The formal development}\label{subsec:formal}

Both implications of the Main Theorem, including the bound
$k+2\ell\le R$ for integer $R$, have been machine-checked in
Lean~4 over mathlib, with no assumed mathematical input
\cite{Whi26}, and so have the consequences of this section that
refer to connection ranks: the growth of the ranks along even
arities, the minimality of the total number of colours of the
reconstructed model, and the formulas expressing its two
dimensions through that minimum and the free-circle value
(Theorem~\ref{theo:modeldims}(a) and (b)), the padding lemma
(Lemma~\ref{lemm:padding}), and Theorem~\ref{theo:prescribed}.
The development states its rank hypothesis with an integer bound
$B$, and the real case of the Main Theorem follows: given a real
$R\ge1$ with $\rk\connmat{f}{t}\le R^{t}$ for all $t$, the integer
$B\coloneqq\lceil R\rceil$ satisfies the same bounds, the
development produces a model with $k+2\ell\le B$, and its
even-arity limit gives
$k+2\ell=\lim_{n\to\infty}(\rk\connmat{f}{2n})^{1/2n}\le R$.
Deligne's theorem \cite[Thm~0.6]{Del02} is proved within the
development in the fibre-functor form of
Theorem~\ref{theo:fibre}, along the splitting-object pattern of
Coulembier \cite[Lem.~3.3.2]{Cou20} rather than Deligne's descent;
the nilpotent-trace step is carried by Schrijver's argument in the
development as in the paper; and the trace zeta theorem of
Appendix~\ref{app:zeta} is formalised in its growth form,
Corollary~\ref{coro:zetagrowth}, with the rationality lemma for
hook-confined sequences (Lemma~\ref{lemm:sfrationality}) in its
general form. The development states Definition~\ref{defi:mpf} through the
values of $h$ on the sorted monomials of
Section~\ref{subsec:paritytransfer}, expanding the wedge of
Definition~\ref{defi:mpf} in that basis with its partner and sorting signs,
which is Definition~\ref{defi:mpf} itself. Its witness for the converse is the
functional $h^{\xi}$ of Lemma~\ref{lemm:letters}, in the
development's terms the coordinate of the star tensor at the
canonical colouring, sorted even colours followed by sorted odd
colours; by that lemma and the Reconstruction Theorem,
$p_{h^{\xi}}=p_{h^{\mathrm{RS}}}=f$, free circles contributing the
factor $k-2\ell$ to both.

\clearpage
\appendix
\section[Rational trace zeta functions from exponential
endomorphism growth]{Rational trace zeta functions\texorpdfstring{\\}{ }from
exponential endomorphism growth}\label{app:zeta}

This appendix proves, by a direct symmetric-function argument,
that the death of a single rectangular Schur idempotent makes the
trace zeta function of every endomorphism rational, with
hook-shaped degree bounds, and that exponential endomorphism
growth forces such a death; the vanishing of nilpotent traces,
proved in
Section~\ref{sec:semisimple} by Schrijver's factorial argument, is
recovered as a corollary. Neither semisimplicity nor nondegeneracy
of the trace pairing is assumed, so the setting is broader than
the one reached for $\cauchy$ in Section~\ref{sec:semisimple}. The
appendix is self-contained: its
inputs are the classical representation theory of symmetric groups
and some symmetric function theory, for which we cite
\cite{Sag01,Mac95,FH91}, and the standard trace calculus of rigid
symmetric categories (cyclicity, multiplicativity over $\otimes$,
and $\tr(c\circ(u\otimes v))=\tr(uv)$; see
e.g.\ \cite[\S4.7]{EGNO15}), which for the connection category
was proved combinatorially in Lemma~\ref{lemm:tracecalc}.

\subsection{Statement}\label{subsec:appstatement}

Throughout the appendix, $\mathcal{C}$ is a rigid symmetric
$\mathbb{C}$-linear monoidal category with
$\End(\unitobj)=\mathbb{C}$, and $Z$ is an object of $\mathcal{C}$.
As in Section~\ref{subsec:fibrefunctor}, $\Phi_{n}$ denotes the
action of $\mathbb{C}[S_{n}]$ on $Z^{\otimes n}$ through the
symmetry, $e_{\lambda}$ the central idempotent of the block of
$\lambda\vdash n$, and $\lambda$ \emph{survives} for $Z$ if
$\Phi_{n}(e_{\lambda})\ne0$ and \emph{dies} otherwise. For
integers $a,b\ge0$,
$\hook{a}{b}\coloneqq\{\lambda:\lambda_{a+1}\le b\}$ is the
\emph{fat hook}, the set of partitions whose Young diagram fits in
the union of $a$ rows and $b$ columns (Figure~\ref{fig:fathook});
a partition lies outside $\hook{a}{b}$ iff it contains the
rectangle $\bigl((b+1)^{a+1}\bigr)$. For a sequence
$\tau=(\tau_{m})_{m\ge1}$ of complex numbers put
\[
\tzeta{\tau}(z)\coloneqq
\exp\Bigl(\sum_{m\ge1}\frac{\tau_{m}}{m}\,z^{m}\Bigr)
\in\mathbb{C}[[z]],
\qquad\text{so that}\qquad
\frac{z\tzeta{\tau}'(z)}{\tzeta{\tau}(z)}=\sum_{m\ge1}\tau_{m}z^{m}.
\]
For $g\in\End(Z)$ we write $\tau_{m}\coloneqq\tr(g^{m})$ and call
$\tzeta{g}\coloneqq\tzeta{\tau}$ the \emph{trace zeta function}
of $g$.

\begin{theo}\label{theo:zeta}
Let $a,b\in\mathbb{N}$, and suppose that the rectangle
$\bigl((b+1)^{a+1}\bigr)$ dies for $Z$. Then for every
$g\in\End(Z)$ the trace zeta function is rational,
\[
\tzeta{g}(z)=\frac{P_{g}(z)}{Q_{g}(z)}
\qquad\text{with}\qquad
\deg P_{g}\le b,\quad \deg Q_{g}\le a,
\]
where $P_{g}$ and $Q_{g}$ are coprime polynomials with constant
term~$1$. Equivalently, there are disjoint multisets of nonzero
complex numbers $\{\alpha_{1},\dotsc,\alpha_{a'}\}$ and
$\{\beta_{1},\dotsc,\beta_{b'}\}$ with $a'\le a$, $b'\le b$ and
\[
\tau_{m}=\sum_{i=1}^{a'}\alpha_{i}^{m}
-\sum_{j=1}^{b'}\beta_{j}^{m}
\qquad(m\ge1).
\]
\end{theo}

We call these disjoint multisets the \emph{reduced
super-spectrum} of $g$; they are the reciprocal-root data of
$Q_{g}$ and $P_{g}$ respectively. When a super vector-space
realisation is available, common even and odd eigenvalues cancel,
and zero eigenvalues and Jordan-block sizes are not recorded.
Exponential
growth of the endomorphism spaces forces the death of a square,
and so is a sufficient condition.

\needspace{10\baselineskip}
\begin{coro}[Exponential growth]\label{coro:zetagrowth}
Suppose that
\begin{equation}\label{eq:growth}
\dim\End(Z^{\otimes n})\le A^{n}\qquad(n\ge 0)
\end{equation}
for a real constant $A\ge 1$, and let $s$ be any integer with
$s>2e\sqrt{A}$. Then the square $(s^{s})$ dies for $Z$, and
Theorem~\ref{theo:zeta} applies with $a=b=s-1$: for every
$g\in\End(Z)$, $\tzeta{g}=P_{g}/Q_{g}$ with
$\deg P_{g}\le s-1$ and $\deg Q_{g}\le s-1$.
\end{coro}

\begin{coro}\label{coro:nilpotent}
Under the hypothesis of Theorem~\ref{theo:zeta} or of
Corollary~\ref{coro:zetagrowth}, every nilpotent $g\in\End(Z)$
has $\tau_{m}=0$ for all $m\ge1$; in particular $\tr(g)=0$.
\end{coro}

The generating object of the connection category illustrates the
degree bounds. For $\genX\in\cauchy$,
Corollary~\ref{coro:rectangle} applied to the reconstructed model
itself says that the rectangle $\bigl((2\ell+1)^{k+1}\bigr)$
dies, so every $g\in\End(\genX)$ has $\tzeta{g}=P_{g}/Q_{g}$ with
$\deg P_{g}\le2\ell$ and $\deg Q_{g}\le k$. The fibre functor
makes this visible: $\fib(g)$ is an even endomorphism of
$V=\mathbb{C}^{k|2\ell}$, and $\tr(g^{m})=\str(\fib(g)^{m})$ is
the $m$-th power sum of its eigenvalues on $V_{0}$ minus that of
its eigenvalues on $V_{1}$.

Section~\ref{subsec:appliterature} places these statements among
their antecedents: the determinantal rationality of Kimura, Kimura
and Takahashi \cite{KKT12}, the moderate-growth theorem of Etingof
and Penneys \cite{EP26}, and the question of Andr\'e and Kahn on
the traces of nilpotents.

\subsection{Fat-hook confinement}\label{subsec:fathook}

We use the standard representation theory of the symmetric group
$S_{n}$ over $\mathbb{C}$ \cite{Sag01,FH91}: the group algebra
decomposes into simple blocks
$\mathbb{C}[S_{n}]=\bigoplus_{\lambda\vdash n}B_{\lambda}$ with
$B_{\lambda}\cong\mathrm{Mat}_{d_{\lambda}}(\mathbb{C})$, indexed
by partitions $\lambda$ of $n$ with irreducible dimensions
$d_{\lambda}$; the central primitive idempotent of $B_{\lambda}$ is
\[
e_{\lambda}
=\frac{d_{\lambda}}{n!}\sum_{\pi\in S_{n}}\chi_{\lambda}(\pi)\,\pi .
\]
The symmetric braiding makes $S_{n}$ act on $Z^{\otimes n}$, giving
the algebra homomorphism $\Phi_{n}$ from $\mathbb{C}[S_{n}]$ to
$\End(Z^{\otimes n})$ of Section~\ref{subsec:appstatement}. Since
$\ker\Phi_{n}$ is a two-sided ideal, hence a sum of blocks, the
whole block $B_{\lambda}$ of a surviving $\lambda$ embeds into
$\End(Z^{\otimes n})$; this is the entire mechanism by which a
growth bound such as \eqref{eq:growth} acts.

\begin{lemm}\label{lemm:blockdeath}
If $\lambda\vdash n$ survives, then
$d_{\lambda}^{2}\le\dim\End(Z^{\otimes n})$; in particular
$d_{\lambda}^{2}\le A^{n}$ when \eqref{eq:growth} holds.
\end{lemm}

\begin{proof}
The block $B_{\lambda}\cong\mathrm{Mat}_{d_{\lambda}}(\mathbb{C})$
embeds into $\End(Z^{\otimes n})$, whose dimension
\eqref{eq:growth} bounds by $A^{n}$.
\end{proof}

\begin{lemm}\label{lemm:upward}
Death is upward closed: if $\lambda\vdash n$ dies and
$\lambda\subseteq\mu\vdash n+m$, then $\mu$ dies.
\end{lemm}

\begin{proof}
We first record the branching containment: if
$\lambda\subseteq\mu$ then the irreducible $U_{\lambda}$ appears in
the restriction of $U_{\mu}$ to $S_{n}$. Indeed, nested Young
diagrams are joined by a chain
$\lambda=\nu_{0}\subset\nu_{1}\subset\dotsb\subset\nu_{m}=\mu$
adding one box at a time; the branching rule
\cite[Thm~2.8.3]{Sag01} gives that $U_{\nu_{i}}$ appears in the
restriction of $U_{\nu_{i+1}}$ to the next symmetric group down,
and ``appears in the restriction'' is transitive by
semisimplicity.

Consequently the element
$e_{\mu}(e_{\lambda}\otimes 1_{m})e_{\mu}$ of the simple block
$B_{\mu}\subseteq\mathbb{C}[S_{n+m}]$ is nonzero: it is the
compression to $U_{\mu}$ of the projection onto the
$\lambda$-isotypic part of the restriction, which is a nonzero
operator by the previous paragraph. If $\mu$ survived, then
$\Phi_{n+m}$ would be injective on $B_{\mu}$, so
$\Phi_{n+m}\bigl(e_{\mu}(e_{\lambda}\otimes1)e_{\mu}\bigr)\ne0$;
but
$\Phi_{n+m}(e_{\lambda}\otimes1)
=\Phi_{n}(e_{\lambda})\otimes\mathrm{id}=0$
since $\lambda$ dies. This contradiction proves the lemma.
\end{proof}

\begin{lemm}\label{lemm:squaredeath}
Suppose that \eqref{eq:growth} holds, and let $s$ be an integer
with $s>2e\sqrt{A}$. Then the square partition $(s^{s})$ dies.
\end{lemm}

\begin{proof}
Every hook length of the square $(s^{s})$ is less than $2s$, so
the hook length formula and $n!\ge(n/e)^{n}$ --- one term of the
series $e^{n}=\sum_{j}n^{j}/j!$ --- give, with $n=s^{2}$,
\[
\log d_{(s^{s})}
\ge \log\frac{(s^{2})!}{(2s)^{s^{2}}}
\ge s^{2}\bigl(\log s-1-\log 2\bigr).
\]
The hypothesis $s>2e\sqrt{A}$ says precisely that
$\log s-1-\log2>\tfrac12\log A$, whence
$d_{(s^{s})}^{2}>A^{s^{2}}$, and Lemma~\ref{lemm:blockdeath} kills
$(s^{s})$.
\end{proof}

\begin{figure}
\centering
\begin{tikzpicture}[scale=0.42]
  \fill[gray!25] (0,0) rectangle (9,-2);
  \fill[gray!25] (0,0) rectangle (3,-9);
  \draw (9,0) -- (0,0) -- (0,-9);
  \draw (9,-2) -- (3,-2) -- (3,-9);
  \node[font=\footnotesize] at (10.1,-1) {$\cdots$};
  \node[font=\footnotesize] at (1.5,-10.1) {$\vdots$};
  \draw[<->] (-0.55,0) -- (-0.55,-2);
  \node[font=\scriptsize,left] at (-0.75,-1) {$a$};
  \draw[<->] (0,0.55) -- (3,0.55);
  \node[font=\scriptsize,above] at (1.5,0.7) {$b$};
  \draw[thick] (0,0) rectangle (5,-5);
  \node[font=\scriptsize] at (4.1,-4.2) {$(s^{s})$};
\end{tikzpicture}
\caption{The fat hook $\hook{a}{b}$ and a dead square.}
\label{fig:fathook}
\end{figure}
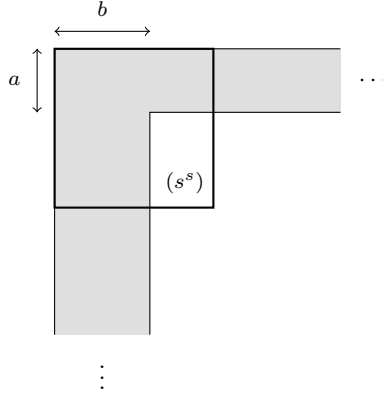
\FloatBarrier

\begin{coro}\label{coro:hookconfine}
If the rectangle $\bigl((b+1)^{a+1}\bigr)$ dies, then every
surviving partition lies in the fat hook $\hook{a}{b}$.
\end{coro}

\begin{proof}
A partition outside $\hook{a}{b}$ contains the rectangle, which
dies; apply Lemma~\ref{lemm:upward}.
\end{proof}

We now convert confinement into a statement about traces. For a
permutation $\pi$ with cycle lengths $c_{1},\dotsc,c_{r}$, the
standard trace calculus recalled at the head of this appendix
gives, by induction on the cycles,
\begin{equation}\label{eq:cyclefact}
\tr\bigl(\Phi_{n}(\pi)\circ g^{\otimes n}\bigr)
=\prod_{i=1}^{r}\tau_{c_{i}} .
\end{equation}
Indeed, each cycle of $\pi$ threads the corresponding tensor
factors of $g^{\otimes n}$ into a single power of $g$, closed
into a trace.
For a partition $\lambda\vdash n$, write $s_{\lambda}[\tau]$ for
the Schur polynomial expressed through the power sums, with the
$m$-th power sum evaluated at $\tau_{m}$; concretely, define the
one-row values
$\onerow{n}[\tau]$ for $n\in\mathbb{Z}$ by
$\onerow{0}[\tau]\coloneqq1$, $\onerow{n}[\tau]\coloneqq0$ for
$n<0$, and
$n\,\onerow{n}[\tau]\coloneqq
\sum_{i=1}^{n}\tau_{i}\onerow{n-i}[\tau]$ by Newton's
identities, and then
$s_{\lambda}[\tau]\coloneqq
\det\bigl(\onerow{\lambda_{i}-i+j}[\tau]\bigr)_{1\le i,j\le r}$,
with $r$ the number of parts of $\lambda$, by Jacobi--Trudi
\cite[Chapter~I]{Mac95}. The one-row values are the complete
homogeneous symmetric functions in the power sums, and Newton's
identities in generating-function form read
\begin{equation}\label{eq:zetaseries}
\tzeta{\tau}(z)=\sum_{n\ge0}\onerow{n}[\tau]\,z^{n}.
\end{equation}

\begin{lemm}[Frobenius formula]\label{lemm:frobenius}
For every $\lambda\vdash n$,
\[
\tr\bigl(\Phi_{n}(e_{\lambda})\circ g^{\otimes n}\bigr)
=d_{\lambda}\cdot s_{\lambda}[\tau].
\]
In particular, if $\lambda$ dies then $s_{\lambda}[\tau]=0$.
\end{lemm}

\begin{proof}
Expand $e_{\lambda}$ and apply \eqref{eq:cyclefact} to each term:
writing $z_{\mu}$ for the order of the centraliser of a
permutation of cycle type $\mu$, so that the number of such
permutations is $n!/z_{\mu}$, and
$\tau_{\mu}\coloneqq\prod_{i}\tau_{\mu_{i}}$ for the product of
trace powers over the parts of $\mu$,
\begin{align*}
\tr\bigl(\Phi_{n}(e_{\lambda})\circ g^{\otimes n}\bigr)
&=\frac{d_{\lambda}}{n!}\sum_{\pi}\chi_{\lambda}(\pi)
  \prod_{i}\tau_{c_{i}(\pi)}
 =d_{\lambda}\sum_{\mu\vdash n}z_{\mu}^{-1}\chi_{\lambda}(\mu)\,
  \tau_{\mu}
 =d_{\lambda}\,s_{\lambda}[\tau],
\end{align*}
the last step being the Frobenius characteristic formula, which
expands $s_{\lambda}$ in the products of power sums over cycle
types with the coefficients $z_{\mu}^{-1}\chi_{\lambda}(\mu)$
\cite[I.7]{Mac95}. If $\lambda$ dies, the left-hand side is the
trace of the zero morphism composed with $g^{\otimes n}$, and
$d_{\lambda}\ne0$.
\end{proof}

Combining Corollary~\ref{coro:hookconfine} with
Lemma~\ref{lemm:frobenius}: if the rectangle
$\bigl((b+1)^{a+1}\bigr)$ dies, then for every $g\in\End(Z)$,
\begin{equation}\label{eq:hookvanishing}
s_{\lambda}[\tau]=0
\qquad\text{for every }\lambda\notin \hook{a}{b}.
\end{equation}

\subsection{Rationality of hook-confined
sequences}\label{subsec:rationality}

The passage from \eqref{eq:hookvanishing} to
Theorem~\ref{theo:zeta} is a piece of classical symmetric function
theory, going back to Borel \cite{Bor1894};
Section~\ref{subsec:appliterature} locates it in the literature.
The proof below is a Jacobi--Trudi null-vector computation, which
yields the hook-shaped degree bounds directly.

\begin{lemm}\label{lemm:sfrationality}
Let $\tau=(\tau_{m})_{m\ge1}$ be a complex sequence with
$s_{\lambda}[\tau]=0$ for every $\lambda\notin\hook{a}{b}$. Then
there are coprime polynomials $P$ and $Q$ with constant term $1$,
$\deg P\le b$ and $\deg Q\le a$, such that
$\tzeta{\tau}(z)=P(z)/Q(z)$; consequently,
factoring $Q(z)=\prod_{i}(1-\alpha_{i}z)$ and
$P(z)=\prod_{j}(1-\beta_{j}z)$, we have
$\tau_{m}=\sum_{i}\alpha_{i}^{m}-\sum_{j}\beta_{j}^{m}$, where
the $\alpha_{i}$ and the $\beta_{j}$ form disjoint multisets of
nonzero numbers of sizes at most $a$ and $b$ respectively.
\end{lemm}

\begin{proof}
Abbreviate $\onerow{n}\coloneqq\onerow{n}[\tau]$. For strictly
decreasing integers
$\rho_{1}>\dotsb>\rho_{a+1}\ge b-a$, set
$\lambda_{i}\coloneqq\rho_{i}+i$. Then
$\lambda_{i}-\lambda_{i+1}=\rho_{i}-\rho_{i+1}-1\ge0$, so
$\lambda$ is a partition, and
$\lambda_{a+1}=\rho_{a+1}+a+1\ge b+1$, so
$\lambda\notin\hook{a}{b}$; the Jacobi--Trudi identity gives
\[
\det\bigl(\onerow{\rho_{i}+j}\bigr)_{i,j\le a+1}
=s_{\lambda}[\tau]=0 .
\]
Hence the vectors
$v_{\rho}\coloneqq(\onerow{\rho+1},\dotsc,\onerow{\rho+a+1})$ for
$\rho\ge b-a$ span a space of dimension at most $a$, and we may
choose a nonzero $c=(c_{1},\dotsc,c_{a+1})$ annihilating that
span under the standard bilinear pairing:
$\sum_{j=1}^{a+1}c_{j}\onerow{\rho+j}=0$ for all $\rho\ge b-a$. Let
$j^{*}$ be maximal with $c_{j^{*}}\ne0$ and normalise
$c_{j^{*}}=1$; the relation reads
\[
\onerow{n}+q_{1}\onerow{n-1}+\dotsb+q_{j^{*}-1}\onerow{n-j^{*}+1}=0
\qquad\text{for all }n\ge N\coloneqq b-a+j^{*} .
\]
Set $Q_{0}(z)\coloneqq1+q_{1}z+\dotsb+q_{j^{*}-1}z^{j^{*}-1}$, of
degree at most $j^{*}-1\le a$. Then
$P_{0}\coloneqq Q_{0}\tzeta{\tau}$ has
vanishing coefficients in every degree $n\ge N$, since by
\eqref{eq:zetaseries} the
coefficient of $z^{\rho+j^{*}}$ in $Q_{0}\tzeta{\tau}$ is exactly the
left-hand side of the relation at $\rho$. Note that $N\ge1$: the
relation at $\rho=-j^{*}$ would read
$\sum_{j}c_{j}\onerow{j-j^{*}}=c_{j^{*}}\onerow{0}=1$, every other index
being negative, so $\rho=-j^{*}$ must lie outside the admissible
range, i.e.\ $-j^{*}<b-a$. Hence
$\deg P_{0}\le N-1=b-a+j^{*}-1\le b$, using $j^{*}\le a+1$.
Divide $P_{0}$ and $Q_{0}$ by their greatest common divisor,
normalised so that the resulting coprime $P$ and $Q$ retain
constant term $1$; factoring
$Q=\prod_{i}(1-\alpha_{i}z)$ and $P=\prod_{j}(1-\beta_{j}z)$ as
in the statement, the
logarithmic derivative identity
$z\tzeta{\tau}'(z)/\tzeta{\tau}(z)=\sum_{m\ge1}\tau_{m}z^{m}$ of
Section~\ref{subsec:appstatement} yields
$\tau_{m}=\sum_{i}\alpha_{i}^{m}-\sum_{j}\beta_{j}^{m}$.
\end{proof}

\begin{lemm}\label{lemm:rigidity}
In the situation of Lemma~\ref{lemm:sfrationality}, if
$\tau_{m}=0$ for all sufficiently large $m$, then $\tau_{m}=0$ for
all $m\ge1$.
\end{lemm}

\begin{proof}
The series $\Sigma(z)\coloneqq\sum_{m\ge1}\tau_{m}z^{m}$ has
finite support, hence is a polynomial, and
$\Sigma=z\tzeta{\tau}'/\tzeta{\tau}=z(P'Q-PQ')/(PQ)$. Thus $PQ$
divides
$z(P'Q-PQ')$. Since $P$ divides $zPQ'$, it divides $zP'Q$; since
$\gcd(P,Q)=1$ and $P(0)=1$ gives $\gcd(P,zQ)=1$, we get
$P\mid P'$, which in characteristic zero forces $P'=0$, so $P$ is
constant. Symmetrically $Q$ is constant. Hence $\tzeta{\tau}=1$,
so $\Sigma=z\tzeta{\tau}'/\tzeta{\tau}=0$, i.e.\ $\tau_{m}=0$ for
all $m$.
\end{proof}

\subsection{Proof of Theorem~\ref{theo:zeta}}\label{subsec:appproof}

\begin{proof}
Let $g\in\End(Z)$. By \eqref{eq:hookvanishing}, the sequence
$\tau_{m}=\tr(g^{m})$ satisfies the hypothesis of
Lemma~\ref{lemm:sfrationality} with the given $a$ and $b$, and
$\tzeta{g}=\tzeta{\tau}$ by definition. The lemma provides the
rational form with its degree bounds and the reduced
super-spectrum.
\end{proof}

\begin{proof}[Proof of Corollary~\ref{coro:zetagrowth}]
Lemma~\ref{lemm:squaredeath} gives the death of the square
$(s^{s})$, which is the rectangle $\bigl((b+1)^{a+1}\bigr)$ for
$a=b=s-1$, and Theorem~\ref{theo:zeta} applies.
\end{proof}

\begin{proof}[Proof of Corollary~\ref{coro:nilpotent}]
If $g$ is nilpotent, say
$g^{M}=0$, then $\tau_{m}=\tr(g^{m})=0$ for $m\ge M$, and
Lemma~\ref{lemm:rigidity} extends the vanishing to all $m\ge1$;
in particular $\tr(g)=\tau_{1}=0$.
\end{proof}

\subsection{Relation to the literature}\label{subsec:appliterature}

The circle of ideas around Theorem~\ref{theo:zeta} has a
substantial history in the theory of motives, which we now locate
precisely; the reader interested only in the graph-theoretic
results may skip this subsection.

An object $Z$ of a pseudo-abelian rigid symmetric category is
\emph{Kimura-finite} if it splits as $Z_{+}\oplus Z_{-}$ with some
exterior power of $Z_{+}$ and some symmetric power of $Z_{-}$
vanishing \cite{Kim05,And04}, and \emph{Schur-finite} if some
Schur functor annihilates it \cite{Del02,Maz04}. Kimura-finiteness
implies Schur-finiteness \cite[Cor.~1.5]{Maz04}, and the converse
fails: O'Sullivan's example of an indecomposable Schur-finite
object that is neither even nor odd appears in
\cite[10.1.1]{AK02} and \cite[Ex.~1.12]{Maz04}. Like
Kimura-finiteness, the growth hypothesis \eqref{eq:growth} forces a
fat-hook vanishing pattern (Lemma~\ref{lemm:squaredeath} and
Corollary~\ref{coro:hookconfine}) ---
at the level of vanishing patterns exactly the notion of a
\emph{bound} in the $\lambda$-ring formalism of Mazza and Weibel
\cite[Def.~4.1]{MW13} --- so the image of $Z$ in the
pseudo-abelian envelope of $\mathcal{C}$ is Schur-finite, with a
rectangular bound. Upward
closure of the vanishing set, automatic for objects
(Lemma~\ref{lemm:upward}; \cite[Prop.~1.4(2)]{Maz04}), fails for
virtual elements of $\lambda$-rings \cite[Ex.~4.4]{MW13}, which
is why their definition builds the ideal in.

Rationality of zeta functions under such hypotheses is classical
in outline. For Kimura-finite objects, rationality (and a
functional equation) is due to Andr\'e, Heinloth and Kahn
\cite{And04,Hei07,Kah09}; for Schur-finite objects of
\emph{semisimple} categories it is proved by Kahn
\cite[Part~I]{Kah09}; and the underlying rationality criterion
through vanishing Hankel determinants goes back to Borel
\cite{Bor1894}, in the form given by Larsen and Lunts
\cite[\S2]{LL04}. Mazza and Weibel assemble exactly this argument
in $\lambda$-rings and prove their virtual splitting principle
for reduced $\lambda$-rings \cite[Prop.~4.14,
Cor.~4.15]{MW13}; since trace sequences live in the ring of big
Witt vectors of $\mathbb{C}$, which is reduced --- in
characteristic zero its ghost coordinates embed it into a product
of copies of $\mathbb{C}$ --- the sequence-level content
of Lemma~\ref{lemm:sfrationality} is theirs.

The categorical statement has an antecedent too. Kimura, Kimura
and Takahashi \cite[Prop.~3.10]{KKT12} prove, for a Schur-finite
object $X$ of a $\mathbb{Q}$-linear pseudo-abelian symmetric
monoidal category, that the motivic zeta function
$\sum_{n}[\SymV^{n}X]\,t^{n}$ with coefficients in the Grothendieck
ring is determinantally rational: by Jacobi--Trudi, the Hankel
determinant $\det([\SymV^{n+i+j}X])_{0\le i,j\le a}$ is
$\pm[S_{\rho}X]$ for a rectangle $\rho$ with $a+1$ rows, and it
vanishes once $\rho$ contains the killing partition.
Theorem~\ref{theo:zeta} can be deduced in their setting as
follows. Pass to the additive idempotent completion
$\mathcal{D}=\Kar(\Add(\mathcal{C}))$, which is again rigid
symmetric with $\End(\unitobj)=\mathbb{C}$, restricting to an
essentially small tensor subcategory containing $Z$ if necessary,
and let $\mathcal{E}$ be the category of pairs $(Y,a)$ with
$Y\in\mathcal{D}$ and $a\in\End(Y)$, its morphisms the morphisms
of $\mathcal{D}$ commuting with the endomorphisms. Direct sums,
tensor products and images of idempotents are formed
componentwise, so $\mathcal{E}$ is again $\mathbb{C}$-linear
pseudo-abelian symmetric monoidal, and $S_{\lambda}(Z,g)=0$
whenever $S_{\lambda}Z=0$. On the split Grothendieck ring of
$\mathcal{E}$, the assignment
$T\bigl([(Y,a)]\bigr)\coloneqq\tr_{Y}(a)$, the trace being taken in
the rigid category $\mathcal{D}$, is a ring homomorphism to
$\mathbb{C}$: cyclicity makes it invariant under isomorphisms of
pairs, and it is additive over direct sums and multiplicative over
tensor products. It sends $[S_{\lambda}(Z,g)]$ to
$s_{\lambda}[\tau]$, by Lemma~\ref{lemm:frobenius} and because the
$d_{\lambda}$ primitive idempotents of the block $B_{\lambda}$ have
equal traces against $g^{\otimes n}$. Applying $T$ to the
vanishings $S_{\mu}(Z,g)=0$ for $\mu\notin\hook{a}{b}$, which
follow from the death of the rectangle by upward closure
(Lemma~\ref{lemm:upward}), gives $s_{\mu}[\tau]=0$ outside
$\hook{a}{b}$, the hypothesis of Lemma~\ref{lemm:sfrationality},
and that lemma gives the numerator and denominator bounds; their
Proposition~3.10 is the same deduction carried out in the
Grothendieck ring. Their Theorem~3.12 shows, by contrast, that in
the Grothendieck ring itself the rationality need not be uniform;
the passage to scalars removes the distinction. What this appendix
adds is the direct objectwise proof of the scalar statement, for
a single object of a rigid symmetric category, with the
rectangular degree bounds stated, and the sufficient criterion for
the death of a rectangle in terms of endomorphism growth
(Corollary~\ref{coro:zetagrowth}).

The vanishing of nilpotent traces, Corollary~\ref{coro:nilpotent},
answers in the present setting a question of Andr\'e and Kahn,
posed at the opening of \cite[\S7.3]{AK02}: ``\emph{Nous ignorons
si un endomorphisme nilpotent est toujours de trace nilpotente,
sans hypoth\`ese suppl\'ementaire}.'' Affirmative answers were
available under Kimura-finiteness \cite[Thm~9.2.2,
Prop~7.3.3]{AK02}, in the presence of a fibre functor to an
abelian rigid target \cite[Prop~7.3.3]{AK02}, and more recently
in braided categories of moderate growth, where Etingof and Penneys
\cite{EP26} prove that non-negligible objects of moderate growth
are rigid, that nilpotent quantum traces vanish, and that a
semisimplification exists. Under exponential endomorphism growth
the answer is affirmative by the factorial argument of
Theorem~\ref{theo:niltrace}, due to Schrijver in the
graph-parameter setting \cite[Prop.~4]{Sch15} and valid verbatim
for a single object; Corollary~\ref{coro:nilpotent} recovers it
as a by-product of rationality, and Theorem~\ref{theo:zeta} adds
the reduced super-spectrum, with degree bounds explicit in the
growth
constant.

The fat hook itself is the territory of Berele and Regev
\cite{BR87}, whose hook Schur functions govern the Schur support
of super vector spaces (compare \cite[Cor.~1.9]{Del02}). The
closest neighbour of our argument in the motivic literature is
the work of Del Padrone and Mazza \cite{DPM05,DPM09}, who use the
Berele--Regev theory to prove nilpotency of endomorphisms
universally of trace zero --- the converse direction to
Corollary~\ref{coro:nilpotent} --- under a sign-property
hypothesis. Finally, Guletski\u{\i}
\cite{Gul10} shows that rationality of the zeta function does not
imply Kimura-finiteness.

\section*{Acknowledgements}

Claude Fable~5, GPT-5.6 Sol Pro and GPT-6 Astra were used
extensively in the development and preparation of this work.

\bibliographystyle{plain}
\bibliography{rs-mpf}

\end{document}